\documentclass[11pt]{amsart}
\usepackage{amssymb}
\usepackage{amscd}
\usepackage{amsmath}
\usepackage{amsmath}
\usepackage[all]{xy}
\usepackage{mathrsfs}

\usepackage{color}
\usepackage{ulem}

\usepackage{bm} 

\theoremstyle{plain}
\newtheorem{theorem}{Theorem}[section]
\newtheorem{lemma}[theorem]{Lemma}
\newtheorem{corollary}[theorem]{Corollary}
\newtheorem{proposition}[theorem]{Proposition}
\newtheorem{conjecture*}{Conjecture}
\newtheorem{theorem*}{Theorem}
\newtheorem{corollary*}{Corollary}

\newtheorem{conjecture}[theorem]{Conjecture}
\theoremstyle{definition}

\newtheorem{hypothesis}[theorem]{Hypothesis}
\theoremstyle{definition}
\newtheorem{definition}[theorem]{Definition}
\newtheorem{remark}[theorem]{Remark}

\font\russ=wncyr10  1
\def\sha{\hbox{\russ\char88}}

\DeclareMathOperator{\Gal}{Gal}

\DeclareMathOperator{\Hom}{Hom}

\DeclareMathOperator{\N}{N}

\DeclareMathOperator{\coker}{coker}

\newcommand{\CC}{\mathbb{C}}

\newcommand{\FF}{\mathbb{F}}
\newcommand{\GG}{\mathbb{G}}

\newcommand{\QQ}{\mathbb{Q}}

\newcommand{\RR}{\mathbb{R}}

\newcommand{\ZZ}{\mathbb{Z}}

\newcommand{\cN}{\mathcal{N}}
\newcommand{\cO}{\mathcal{O}}

\newcommand{\fp}{\mathfrak{p}}

\newcommand{\fz}{\mathfrak{z}}

\newcommand{\catname}[1]{\textnormal{{\textsf{#1}}}}

\newcommand{\DR}{\catname{R}}
\newcommand{\DL}{\catname{L}}

\newcommand{\rgamma}{\DR\Gamma}

\newcommand{\lotimes}{\otimes^{\DL}}

\begin{document}

\title[]{On derivatives of Kato's Euler system\\
 and the Mazur-Tate conjecture}

\author{David Burns, Masato Kurihara and Takamichi Sano}

\begin{abstract} We provide a new interpretation of the Mazur-Tate Conjecture and then use it to obtain the first (unconditional) theoretical evidence in support of the conjecture for elliptic curves of strictly positive rank. \end{abstract}

\address{King's College London,
Department of Mathematics,
London WC2R 2LS,
U.K.}
\email{david.burns@kcl.ac.uk}

\address{Keio University,
Department of Mathematics,
3-14-1 Hiyoshi\\Kohoku-ku\\Yokohama\\223-8522,
Japan}
\email{kurihara@math.keio.ac.jp}

\address{Osaka City University,
Department of Mathematics,
3-3-138 Sugimoto\\Sumiyoshi-ku\\Osaka\\558-8585, 
Japan}
\email{sano@sci.osaka-cu.ac.jp}

\maketitle

\tableofcontents

\section{Introduction} \label{Intro}

\subsection{Discussion of the main results} 
Let $E$ be an elliptic curve over $\QQ$ and write $r$ for the rank of its Mordell-Weil group $E(\QQ)$. 

In \cite{MT} Mazur and Tate formulated a `refined conjecture of Birch and Swinnerton-Dyer type' for $E$ relative to a finite real abelian extension $F$ of $\QQ$. Set $G := \Gal(F/\QQ)$ and write $I(G)$ for the augmentation ideal of the integral group ring $\ZZ[G]$. Then, under the assumption that $E$ has good reduction at ramifying primes in $F/\QQ$, the conjecture of Mazur and Tate predicts, roughly speaking, that an element of $\QQ[G]$ constructed from modular symbols associated to $E$ belongs to the $r$-th power of $I(G)$ and, further, that its image in the quotient group $I(G)^r/I(G)^{r+1}$ is equal to the product of the discriminant of a canonical $G$-valued pairing on $E(\QQ)$ defined using the geometrical theory of biextensions by a suitable ratio of the orders of finite groups that occur in the Birch and Swinnerton-Dyer Conjecture for $E$ over $\QQ$. For convenience, we refer to the prediction that the modular element belongs to $I(G)^r$ and to the predicted formula for its image in $I(G)^r/I(G)^{r+1}$ as respectively the `order of vanishing' and `leading-term' components of the Mazur-Tate Conjecture.

If $r = 0$, then the full Mazur-Tate Conjecture (which is stated precisely as Conjecture \ref{mtconj} below) is easily seen to be equivalent to the original conjecture of Birch and Swinnerton-Dyer (see Remark \ref{rem bsd}) and if $r > 0$, then the  order of vanishing component of the conjecture has proved  amenable to analysis via Euler system methods by using Kato's zeta elements (see the recent article of Ota \cite{ota}). However, if $r > 0$, then the leading-term component of the conjecture has remained stubbornly mysterious and, even now, there is no conceptual understanding of it or theoretical evidence for it and the only supporting numerical evidence is for the case $r =1$ (for details of which see, for example, the recent article of Portillo-Bobadilla \cite{portillo}). 

Notwithstanding these difficulties, Mazur and Tate's celebrated conjecture has been very influential and led to the study of a range of similar conjectures, both in the setting of elliptic curves (for example, by Darmon and by Bertolini and Darmon) and in the setting of the multiplicative group (for example, by Gross, by Tate, by Darmon and by Mazur and Rubin). 

In the sequel we assume $r > 0$. In this case, we formulated in an earlier article an explicit refinement of  Perrin-Riou's conjecture (from \cite{PR}) concerning the logarithm of Kato's zeta element at a fixed prime $p$. We recall that the `Generalized Perrin-Riou Conjecture' formulated in \cite{bks4} predicts a precise congruence relation between a natural $(r-1)$-st order `Darmon derivative' of the zeta element at $p$ of $E$ over an arbitrary real abelian field and the value at $s=1$ of the $r$-th derivative of the Hasse-Weil $L$-function $L(E,s)$ of $E$ over $\QQ$. In particular, the special case of the conjecture relating to subfields of the cyclotomic $\ZZ_p$-extension of $\QQ$ (which is stated explicitly as Conjecture \ref{GPR20} below) is known to be equivalent in the case $r = 1$ to Perrin-Riou's original conjecture (see \cite[Rem. 4.10]{bks4}) and in the case of general $r$ to follow, up to multiplication by an element of $\ZZ_p^\times$, from the validities of the relevant cases of the Birch and Swinnerton-Dyer conjecture and generalized Iwasawa main conjecture (see \cite[Th. 7.3]{bks4}), and hence to follow essentially from the relevant case of the equivariant Tamagawa number conjecture.

\vspace{5mm}

Building on this earlier approach, in the current article we now develop techniques that allow us to prove results of the following sort.  

To state the result we write ${\rm Tam}(E)$ for the product of the Tamagawa factors of $E$ at each prime of bad reduction (see \S\ref{bsd sec}). In addition, for each prime $p$ we regard the group $E[p]$ of $\overline{\QQ}$-rational points of $E$ of order dividing $p$ as a module over $\Gal(\overline{\QQ}/\QQ)$ in the natural way, and we write $E^{\rm ns}(\FF_p)$ for the group of non-singular $\FF_p$-rational points of the reduction of $E$ modulo $p$. 


\begin{theorem} \label{TheoremIntroduction} Let $E$ be an elliptic curve over $\QQ$ and $F$ a finite real abelian extension of $\QQ$ that satisfies all of the following conditions.   
\begin{itemize}
\item[(a)] The conductor $m$ of $F$ is square-free and coprime to the conductor $N$ of $E$.
\item[(b)] $[F:\QQ]$ is coprime to $6\cdot m\cdot N\cdot {\rm Tam}(E)$ and to $\#E^{\rm ns}(\FF_p)$ for every prime divisor $p$ of $m\cdot N\cdot [F:\QQ]$. 
\item[(c)] For every prime divisor $p$ of $[F:\QQ]$ the action of $\Gal(\overline{\QQ}/\QQ)$ on $E[p]$ is surjective.
\end{itemize}
Then the Mazur-Tate Conjecture is valid for $E$ relative to $F/\QQ$ whenever both of the following conditions are satisfied: 
\begin{itemize}
\item[(d)] $E$ validates the Birch and Swinnerton-Dyer Conjecture over $\QQ$;
\item[(e)] If $r > 0$, then for every prime divisor $p$ of $[F:\QQ]$, the Generalized Perrin-Riou Conjecture is valid for $E$ relative to the cyclotomic $\ZZ_p$-extension of $\QQ$ (see Conjecture \ref{GPR20} for the Generalized Perrin-Riou Conjecture).
\end{itemize}
\end{theorem}

This result shows that, under certain mild technical hypotheses, the `refined' nature of the Mazur-Tate Conjecture in comparison with the Birch and Swinnerton-Dyer Conjecture is accounted for by our Generalized Perrin-Riou conjecture in \cite{bks4}. Since the latter conjecture can itself be deduced from the validity of certain standard conjectures (as recalled above), Theorem \ref{TheoremIntroduction} therefore gives a new conceptual interpretation of the Mazur-Tate Conjecture and, even better, can be combined with existing results on the conjectures of Birch and Swinnerton-Dyer and Perrin-Riou to obtain concrete theoretical evidence in support of the Mazur-Tate Conjecture for elliptic curves with strictly positive rank. 
  
In this way, for example, we are able to give unconditional evidence in support of the Mazur-Tate Conjecture of the sort described in the next result. Before stating this result we recall that if $r = 1$, then the group 
$I(G)^r/I(G)^{r+1}$ is naturally isomorphic to $G$, and we say that an equality $x=y$ of elements of $G$ is `valid up to an automorphism of $G$' if there exists an automorphism $\alpha$ of $G$ such that $x = \alpha(y)$. 
  
\begin{corollary}\label{cor1} Assume that $E$ and $F/\QQ$ satisfy the conditions (a), (b) and (c) in Theorem \ref{TheoremIntroduction}. Assume also that $L(E,s)$ vanishes to order one at $s=1$ (which implies $r=1$ by Gross, Zagier and 
Kolyvagin), that the conductor of $E$ is square-free and that $E$ has supersingular reduction at each prime divisor of $[F:\QQ]$. 

Then the order of vanishing component of the Mazur-Tate Conjecture for $E$ and $F/\QQ$ is valid and the leading-term component of the Mazur-Tate Conjecture for $E$ and $F/\QQ$ is valid up to an automorphism of $\Gal(F/\QQ)$. 

Further, if the Mazur-Tate Conjecture is not valid for $E$ and $F/\QQ$, then the Birch and Swinnerton-Dyer Conjecture is not valid for $E$ over $\QQ$.  
\end{corollary}   

This result gives the first theoretical (and unconditional) evidence in support of the leading term component of the Mazur-Tate Conjecture for any elliptic curve of positive rank. Its proof will be given in \S\ref{deduction cor1 section} and combines calculations that are used in the proof of Theorem \ref{TheoremIntroduction} with recent results of B\"uy\"ukboduk \cite{buyuk perrin} on Perrin-Riou's conjecture and of Jetchev, Skinner, and Wan \cite{JSW} on the Birch and Swinnerton-Dyer Conjecture. In fact, for prime divisors $p$ of $[F:\QQ]$ at which $E$ has ordinary reduction, our approach can also be similarly combined with the recently announced results of Bertolini and Darmon \cite{BD}, and of B\"uy\"ukboduk, Pollack and Sasaki \cite{BPS}, concerning Perrin-Riou's conjecture to obtain further theoretical evidence in support of the Mazur-Tate Conjecture. 

It is interesting to note that, whilst the Mazur-Tate Conjecture predicts an equality in the group $I(G)^r/I(G)^{r+1}$, the corresponding case of the Generalized Perrin-Riou Conjecture predicts an equality in $E(\QQ)\otimes I(G)^{r-1}/I(G)^r$ and so is in natural sense `finer'. Another interesting feature of Theorem \ref{TheoremIntroduction} is that, for a given curve $E$, the Generalized Perrin-Riou Conjecture over the cyclotomic $\ZZ_p$-extension $\QQ_\infty$ of $\QQ$ can be used to prove the $p$-primary component of the Mazur-Tate Conjectures relative to fields $F$ that are disjoint from $\QQ_\infty$.

To prove Theorem \ref{TheoremIntroduction} we shall in fact first formulate for each prime $p$ an `algebraic' analogue of the Generalized Perrin-Riou Conjecture for $E$ relative to the cyclotomic $\ZZ_p$-extension of $\QQ$ (see Conjecture \ref{GPR2}). This conjecture is we feel of some independent interest and, in particular, can be seen to be equivalent to the relevant case of the Generalized Perrin-Riou Conjecture  precisely when $E$ validates the Birch and Swinnerton-Dyer Conjecture over $\QQ$. 
Theorem \ref{TheoremIntroduction} thereby follows directly from Theorem \ref{MainResult} which asserts that, under the stated technical hypotheses, Conjecture \ref{GPR2} implies the $p$-primary component of the Mazur-Tate Conjecture for $E$ relative to $F/\QQ$.  

Our main task is therefore to prove Theorem \ref{MainResult} and the argument we use for this has several key steps (that are split across \S\ref{state} and \S\ref{sec pf}). Firstly, we shall apply the theory of equivariant Kolyvagin and Stark systems (as developed by Sakamoto and the first and third authors in \cite{bss}, and already used in this setting by Kataoka in \cite{kataoka}) to Kato's zeta elements at $p$ in order to study certain canonical $p$-adic `determinantal zeta elements' (see Definition \ref{det zeta def}). Then we shall show Conjecture \ref{GPR2} implies that the $p$-adic determinantal zeta element associated to $\QQ$ is equal to a $p$-adic `algebraic Birch and Swinnerton-Dyer element' that arises in the approach of \cite{bks4} (see Proposition \ref{main0}). Finally, we shall show, by explicit computation, that the latter equality implies the validity of the $p$-primary component of the Mazur-Tate Conjecture (see Theorem \ref{main}). We remark that this last computation relies on the precise relation between modular elements and Kato's zeta elements (as previously discussed by the second author in \cite{kuriharass}, by Otsuki in \cite{otsuki} and by Ota in \cite{ota}), on an explicit description of the relevant Bloch-Kato Selmer complexes and on a Galois-cohomological interpretation of the biextension-pairing of Mazur and Tate in terms of Bockstein homomorphisms associated to Bloch-Kato Selmer complexes that is proved by Macias-Castllo and the first author in \cite{bmc} (and relies on earlier cohomological calculations of Tan and of Bertolini and Darmon).   

Finally, we note that, whilst the hypothesis that $m$ is square-free (in Theorem \ref{TheoremIntroduction}) is equivalent to the condition that $F/\QQ$ is tamely ramified, we believe that our general approach should also work in the setting of arbitrary real abelian fields $F$.

\subsection{General notation}

For the reader's convenience we collect together some notations and conventions that will be used in the sequel (and are, in  general, consistent with those used in \cite{bks4}). 

We write $\overline \QQ$ for the algebraic closure of $\QQ$ in $\CC$ and regard any algebraic extension of $\QQ$ as a subfield of $\overline \QQ$. We set $G_\QQ:=\Gal(\overline \QQ/\QQ)$. For each natural number $n$ we also set 
\[ \zeta_n := e^{2\pi i/n} \in \overline{\QQ}.\]


For an abelian group $M$, we write $M_{\rm tors}$ for its torsion subgroup and $M_{\rm tf}$ for the quotient of $M$ by $M_{\rm tors}$ (which we regard as a subgroup of $\QQ\otimes_\ZZ M$ in the natural way). For a prime number $p$, we denote by $M[p]$ and $M[p^\infty]$ the subgroups of $M$ comprising elements that are respectively annihilated by $p$ and by some power of $p$. We also write $\ZZ_{(p)}$ for the localization of $\ZZ$ at $p$ and $\ZZ_p$ and $\QQ_p$ for the ring of $p$-adic integers and field of $p$-adic rationals (so that $\ZZ_{(p)} = \QQ\cap \ZZ_p \subset \QQ_p$). 

If $M$ is a module over a commutative ring $R$, we set
$$M^\ast:=\Hom_R(M,R).$$
If $M$ is a free $R$-module with basis $\{x_1,\ldots,x_r\}$, then the dual basis of $M^\ast$ is denoted by $\{x_1^\ast,\ldots,x_r^\ast\}$. 

For a $\ZZ_p$-module $M$, the Pontryagin dual is defined by 
$$M^\vee:=\Hom_{\ZZ_p}(M,\QQ_p/\ZZ_p).$$

For a finite group $\Gamma$, we set $\widehat \Gamma:= \Hom(\Gamma,\CC^\times) = \Hom(\Gamma,\overline{\QQ}^\times)$. For $\chi \in \widehat \Gamma$, we define a primitive idempotent of $\overline{\QQ}[\Gamma]$ by setting 
$$e_\chi:=\frac{1}{\# \Gamma}\sum_{\sigma \in \Gamma}\chi(\sigma)\sigma^{-1}.$$

If $E$ has good reduction at a rational prime $\ell$, then we write $E(\FF_\ell)$ in place of 
$E^{\rm ns}(\FF_\ell)$. We also use several standard notations for elliptic curves such as $E_0$, $E_1$ and, for any finite set $S$ of prime numbers, $L_S(E,s)$, $L_S(E,\chi,s)$, etc. 

For a number field $F$ and a prime number $p$, we set
$$F_p:=F\otimes_\QQ \QQ_p \simeq \bigoplus_{\fp \mid p} F_\fp,$$
where in the direct sum $\fp$ runs over all $p$-adic places of $F$. 

We use standard notations for Galois (\'etale) cohomology complexes such as $\rgamma(\cO_{F,S},-)$, $\rgamma_f(F,-)$, etc. 
The notation $E_1(F_p)$ indicates $\bigoplus_{\fp \mid p} E_1(F_\fp)$ and we use similar notation to denote  `semi-local' Bloch-Kato complexes such as $\rgamma_f(F_p,-)$, $\rgamma_{/f}(F_p,-)$, etc.

\section{Review of the Mazur-Tate Conjecture}\label{subsec MazurTate}

In this section, we quickly review the statement of the Birch and Swinnerton-Dyer Conjecture for elliptic curves over $\QQ$ and then describe a reformulation of the Mazur-Tate Conjecture that is convenient for our approach. 

\subsection{The Birch and Swinnerton-Dyer Conjecture}\label{bsd sec}

Let $E$ be an elliptic curve over $\QQ$ for which the Tate-Shafarevich group $\sha(E/\QQ)$ of $E$ over $\QQ$ is finite. 

We write $N$ for the conductor of $E$, consider the product of Tamagawa factors 
$${\rm Tam}(E):=\prod_{\ell \mid N}\# (E(\QQ_\ell)/E_0(\QQ_\ell) )$$
and define the `algebraic Birch and Swinnerton-Dyer constant' for $E$ over $\QQ$ to be the rational number
\begin{equation*}\label{bsd cons} \mathcal{L}^{\rm alg} = \mathcal{L}(E)^{\rm alg} := \frac{\# \sha(E/\QQ)\cdot{\rm Tam}(E)}{(\# (E(\QQ)_{\rm tors}))^2}.\end{equation*}

We next fix a minimal Weierstrass model of $E$ over $\ZZ$ and write $\omega$ for the corresponding N\'eron differential. We also fix 
a generator $\gamma$ of $H_1(E(\CC),\ZZ)^+$, write $c_\infty$ for the number of connected components of $E(\RR)$ and define the real period of $E$ by setting 
$$\Omega^+ = \Omega^+_\omega :=c_\infty\left| \int_\gamma \omega \right|.$$
(We note that this period is {\it twice} the period $\Omega_E^+$ defined in \cite{MT}.) We finally write ${\rm Reg}^{\rm NT}$ for the classical N\'{e}ron-Tate regulator of $E$.

We then recall that the Birch and Swinnerton-Dyer Conjecture for $E$ over $\QQ$ predicts that if one defines the `analytic Birch and Swinnerton-Dyer constant' to be the real number  
\[ \mathcal{L}^{\rm an} = \mathcal{L}(E)^{\rm an} := \frac{L^{(r)}(E,1)}{r!\cdot \Omega^+\cdot  {\rm Reg}^{\rm NT}},\]
where we set $r:={\rm rank}(E(\QQ))$ and $L^{(r)}(E,s)$ denotes the $r$-th derivative of $L(E,s)$, then there should be an equality  
\begin{equation}\label{bsd conj} \mathcal{L}^{\rm an} = \mathcal{L}^{\rm alg}.\end{equation}

\begin{remark}\label{truncate remark} The value at $s=1$ of the Euler factor of $E$ at a rational prime $\ell$ is equal to 
$\# E^{\rm ns}(\FF_\ell)/\ell$. Thus, if for any finite set of rational primes $\Sigma$ one defines a real number $\mathcal{L}^{\rm an}_\Sigma = \mathcal{L}_\Sigma(E)^{\rm an}$ just as with $\mathcal{L}^{\rm an}$ except that $L(E,s)$ is replaced by the 
$\Sigma$-truncated Hasse-Weil $L$-function $L_\Sigma(E,s)$, then the conjectural equality 
(\ref{bsd conj}) is valid if and only if $\mathcal{L}^{\rm an}_\Sigma$ is equal to the $\Sigma$-truncated algebraic Birch and Swinnerton-Dyer constant that is defined by setting 
\begin{equation}\label{euler factor} \mathcal{L}^{\rm alg}_\Sigma := \mathcal{L}^{\rm alg}\cdot\prod_{\ell \in \Sigma} 
\frac{\# E^{\rm ns}(\FF_\ell)}{\ell} = \frac{\# \sha(E/\QQ)\cdot{\rm Tam}(E)}{(\# (E(\QQ)_{\rm tors}))^2} \cdot \prod_{\ell \in \Sigma} 
\frac{\# E^{\rm ns}(\FF_\ell)}{\ell} .\end{equation}
\end{remark}

\subsection{Modular elements and the Mazur-Tate Conjecture} We fix a finite real abelian extension $F$ of $\QQ$ of conductor $m$, so that $m$ is the smallest natural number for which $F$ is contained in the maximal real subfield $\QQ(\zeta_m)^+$ of $\QQ(\zeta_m)$. We consider the Galois groups
$$G_m:=\Gal(\QQ(\zeta_m)/\QQ), \ G_m^+:=\Gal(\QQ(\zeta_m)^+/\QQ) \text{ and }G:=\Gal(F/\QQ).$$ 

We fix a prime $p$ and consider the following assumptions on the data $E, F$ and $p$. 

\begin{hypothesis}\label{hyp1}\
\begin{itemize}
\item[(i)] $F/\QQ$ is tamely ramified, i.e., $m$ is square-free. 
\item[(ii)] $E$ has good reduction at every prime divisor of $m$, i.e., $(m,N)=1$.
\item[(iii)] $p$ does not divide $6mN \cdot \# (E(\QQ)_{\rm tors})\cdot  {\rm Tam}(E)\cdot \prod_{\ell \mid pmN} \# E^{\rm ns}(\FF_\ell)$. 
\item[(iv)] The action of $G_\QQ$ on $E[p]$ is surjective.
\end{itemize}
\end{hypothesis}


We define a finite set of primes 
$$S:=
\{\ell \mid \text{$F/\QQ$ is ramified at $\ell$ or $E$ has bad reduction at $\ell$}\}
=\{\ell \mid mN\}$$
and note that Hypothesis \ref{hyp1}(iii) implies that $S$ {\it does not} contain $p$. 

We write 
$$\theta_{F,S} = \theta(E)_{F,S} \in \QQ[G]$$
for the `$S$-truncated modular element' that is characterized by the interpolation property  
\begin{eqnarray}\label{char}
\chi(\theta_{F,S})=\tau_m(\chi) \frac{L_{S } (E,\chi^{-1},1)}{\Omega^+} \text{ for every $\chi \in \widehat G$},
\end{eqnarray}
where $\tau_m(\chi)$ is the Gauss sum
$$\tau_m(\chi):=\sum_{\sigma \in G_m} \chi(\sigma)\zeta_m^\sigma.$$
Here we regard $\chi$ as a character of $G_m$ via the natural surjection $G_m \twoheadrightarrow G$. 
Note that this $S$-truncated modular element is slightly different 
from the original element defined in \cite{MT}, 
but the comparison is not difficult. 
One can construct $\theta_{F,S}$ directly from modular symbols, namely from 
the integrals of the corrsponding modular form, but 
we will later give a construction from Kato's Euler system 
(see Proposition \ref{modularkato} below). Since, under Hypothesis \ref{hyp1}, $p$ does not divide $\#(E(\QQ)_{\rm tors})$ one knows that 
$$\theta_{F,S} \in \ZZ_{(p)}[G]$$
and, in the sequel, we use this containment to regard $\theta_{F,S}$ as an element of $\ZZ_p[G]$. 

We write $\epsilon_G$ for the $\ZZ_p$-linear `augmentation' map $\ZZ_p[G]\to \ZZ_p$ that sends each element of $G$ to $1$ and write 
$$I_F:=\ker(\epsilon_G)$$
for the associated augmentation ideal of $\ZZ_p[G]$. We write $I$ for $I_{F}$ when there is no confusion. 
We recall, in particular, that there exists a canonical isomorphism of abelian groups
\begin{eqnarray}\label{aug isom}
G \otimes \ZZ_{p} \simeq I/I^2; \ \sigma \mapsto \sigma -1.
\end{eqnarray}
Let
$$\langle -,- \rangle_m: E(\QQ) \times E_S(\QQ) \to G_m^+$$
be the Mazur-Tate pairing constructed in \cite[Chap. II]{MT}, where
$$E_S(\QQ):=\ker\left(E(\QQ) \to \bigoplus_{\ell \mid m} E(\FF_\ell) \oplus \bigoplus_{\ell \mid N} E(\QQ_\ell)/E_0(\QQ_\ell)\right).$$
Composing this pairing with the natural surjection
$$G_m^+\twoheadrightarrow G \otimes \ZZ_{p} 
\stackrel{(\ref{aug isom})}{\simeq} I/I^2,$$
we obtain the pairing
$$\langle -,- \rangle_F: E(\QQ) \times E_S(\QQ) \to I/I^2.$$
Hypothesis \ref{hyp1}(iii) ensures that $\ZZ_p\otimes_\ZZ E_S(\QQ) =\ZZ_p \otimes_\ZZ E(\QQ)$, so we obtain
\begin{eqnarray}\label{mt pair}
\langle -,- \rangle_F: (\ZZ_p\otimes_\ZZ E(\QQ))\times (\ZZ_p\otimes_\ZZ E(\QQ)) \to I/I^2.
\end{eqnarray}
In \S\ref{section bmc} below we will describe an alternative construction of this pairing in terms of a natural Bockstein homomorphism on Galois cohomology.

We set $r:={\rm rank}(E(\QQ))$ and define the Mazur-Tate regulator
$$R_F = R(E)_F \in I^r/I^{r+1}$$
to be the discriminant of the pairing $\langle -,- \rangle_F$, i.e., 
\begin{equation}\label{MT reg def} R_F:= \det(\langle x_i, x_j \rangle_F)_{1\leq i,j \leq r}\end{equation}
with $\{x_1,\ldots,x_r\}$ a basis of $E(\QQ)_{\rm tf}$.

Finally, let $\nu(m)$ denote the number of prime divisors of $m$. 

We can now recall (the $p$-component of) the Mazur-Tate Conjecture. We note that this statement uses the algebraic Birch and Swinnerton-Dyer constant $\mathcal{L}^{\rm alg}_S$ defined in (\ref{euler factor}). 

\begin{conjecture}[Mazur-Tate Conjecture]\label{mtconj}
We have
$$\theta_{F,S} \in I_F^r$$
and 
$$\theta_{F,S} =(-1)^{\nu(m)} \mathcal{L}^{\rm alg}_S\cdot R_F \text{ in }I_F^r/I_F^{r+1}.$$
\end{conjecture}

\begin{remark}
 Hypothesis \ref{hyp1}(iii) implies that $p$ does not divide $m\cdot N\cdot\# (E(\QQ)_{\rm tors})$ and hence that the rational number 
 $\mathcal{L}^{\rm alg}_S$ 
 belongs to $\ZZ_{(p)}$. This shows that the right hand side of the `leading term formula' in Conjecture \ref{mtconj} is well-defined. 
\end{remark}

\begin{remark}\label{rem bsd}
When $r=0$, Conjecture \ref{mtconj} is equivalent to the classical leading term formula predicted by Birch and Swinnerton-Dyer. To see this we note that the interpolation formula (\ref{char}) combines with the elementary equality $\tau_m(1)=(-1)^{\nu(m)}$ to imply the image of $\theta_{F,S}$ in $\ZZ_p[G]/I \simeq \ZZ_p$ is equal to $(-1)^{\nu(m)}L_S(E,1)/\Omega^+$ and hence that Conjecture \ref{mtconj} is equivalent in this case to an equality $L_S(E,1)/\Omega^+ = \mathcal{L}^{\rm alg}_S$. This equality coincides precisely with the relevant case of the Birch-Swinnerton-Dyer formula, as recalled in Remark \ref{truncate remark}. \end{remark}

\begin{remark} \label{pextension} It suffices to verify Conjecture \ref{mtconj} after replacing $F$ by its maximal subextension $F^{(p)}$ of $p$-power degree over $\QQ$. This is true since it is enough (following Remark \ref{rem bsd}) to verify Conjecture \ref{mtconj} in the case $r>0$ and, in this case, the natural projection map $\ZZ_{p}[G] \to \ZZ_{p}[\Gal(F^{(p)}/\QQ)]$ sends $\theta_{F,S}$ to $\theta_{F^{(p)},S}$ and also induces an isomorphism of finite groups  
$I_{F}^{r}/I_{F}^{r+1} \simeq I_{F^{(p)}}^{r}/I_{F^{(p)}}^{r+1}$. 
\end{remark}

\begin{remark} \label{RemarkEquivalence}
Under the given hypotheses on $p$, Conjecture \ref{mtconj} is equivalent to the `$p$-primary component' of the original Mazur-Tate conjecture \cite[Conj. 4]{MT}. To check this, it is enough to show Conjecture \ref{mtconj} is equivalent to Darmon's reformulation \cite[Conj. 2.4]{darmoneuler} of the $p$-component of the Mazur-Tate Conjecture and, following Remark \ref{pextension}, we may, and will, assume that $\#G$ is a power of $p$. To proceed we write   
$$\theta_{F,m} \in \QQ[G]$$
for $c(\varphi)^{-1}$ times the projection of the element $\theta_m^{\rm MT}$ of $\QQ[G_m^+]$, where $\theta_m^{\rm MT}$ is defined in \cite[\S 2.3.1]{darmoneuler} by using a modular parameterization $\varphi: X_0(N)\to E$ with Manin constant $c(\varphi)$. We recall (from \cite[Prop. 2.3]{darmoneuler}) that the modular element $\theta_{F,m}$ is characterized by the explicit interpolation formula 
\begin{eqnarray}\label{char2}
\chi(\theta_{F,m})=\tau_m(\chi)\frac{m}{f_\chi} \frac{L_m(E,\chi^{-1},1)}{\Omega^+} \text{ for every }\chi \in \widehat G,
\end{eqnarray}
where $f_\chi$ denotes the conductor of $\chi$ and $L_m(E,\chi,s)$ the $m$-truncated $L$-function. 

Then, for this modular element, \cite[Conj. 2.4]{darmoneuler} predicts the following equality
\begin{eqnarray}\label{dmt}
\theta_{F,m} =  (-1)^{\nu(m)} \# \sha(E/\QQ)\cdot J_S\cdot R_S \,\, \text{ in }I^r/I^{r+1},
\end{eqnarray}
where $J_S$ is the order of the finite group
$$\coker\left(E(\QQ) \to \bigoplus_{\ell \mid m} E(\FF_\ell) \oplus \bigoplus_{\ell \mid N} E(\QQ_\ell)/E_0(\QQ_\ell)\right) $$
and $R_S$ is the discriminant in $I^r/I^{r+1}$ of the pairing
$$\langle -, -\rangle_F: E(\QQ ) \times E_S(\QQ) \to I/I^2$$
(in the sense of \cite[(2.5)]{MT}). In particular, these definitions imply directly that 
\begin{eqnarray}\label{reg}
J_S\cdot R_S = \left( \prod_{\ell \mid m}\# E(\FF_\ell)\right)\frac{{\rm Tam}(E)}{(\# (E(\QQ)_{\rm tors}))^2}\cdot R_F.
\end{eqnarray}

On the other hand, Hypothesis \ref{hyp1}(ii) and (iii) combine to imply that the product ${\rm Eul}_{F/\QQ,N}$ of the $G$-equivariant Euler factors at each prime divisor $\ell$ of $N$ belongs to $\ZZ_{(p)}[G]$. It is also clear that the augmentation map $\epsilon_G$ sends ${\rm Eul}_{F/\QQ,N}$ to $\prod_{\ell \mid N}\bigl(\# E^{\rm ns}(\FF_\ell)/\ell\bigr)$ and hence, since the latter element is a unit of $\ZZ_p$ (as a consequence of Hypothesis \ref{hyp1}(iii)) and $\#G$ is a power of $p$, that ${\rm Eul}_{F/\QQ,N}$ is a unit in 
$\ZZ_{(p)}[G]$. 

In addition, the interpolation formulas (\ref{char}) and (\ref{char2}) together imply an equality 
\begin{eqnarray}\label{eul}
{\rm Eul}_{F/\QQ,N}\cdot \theta_{F,m} = \left(\sum_{\chi \in \widehat G}\frac{m}{f_\chi}e_\chi \right)\theta_{F,S},
\end{eqnarray}
We also note that  
$$\sum_{\chi \in \widehat G}\frac{m}{f_\chi}e_\chi\,\, \text{ belongs to }\,\, \ZZ_{(p)}[G]^\times$$
(by \cite[Prop. 3.1]{bleymodular}) and that $\epsilon_G$ sends this element to $m$. 

Now suppose that (\ref{dmt}) holds. Then in $I^r/I^{r+1}$ one computes that 
\begin{align*}
\theta_{F,S} \stackrel{(\ref{eul})}{=}&\, m^{-1} \left( \prod_{\ell \mid N}\frac{\# E^{\rm ns}(\FF_\ell)}{\ell}\right)\cdot\theta_{F,m} \\
 \stackrel{(\ref{dmt})}{=}&\, (-1)^{\nu(m)}m^{-1} \left( \prod_{\ell \mid N}\frac{\# E^{\rm ns}(\FF_\ell)}{\ell}\right)\cdot \# \sha(E/\QQ)\cdot J_S\cdot R_S \\
\stackrel{(\ref{reg})}{=}&\, (-1)^{\nu(m)}\mathcal{L}^{\rm alg}_S\cdot R_F.
\end{align*}
This shows that the formula predicted in Conjecture \ref{mtconj} is implied by (\ref{dmt}). In addition, since ${\rm Eul}_{F/\QQ,N}$ is a unit in $\ZZ_{(p)}[G]^\times$, the same argument also shows that the equality (\ref{dmt}) is implied by that in Conjecture \ref{mtconj}. 

\end{remark}

\begin{remark} Assume (following Remark \ref{rem bsd}) that $r > 0$. Then, in this case, the Mazur-Tate Conjecture is valid if and only if its $p$-primary component is valid for every prime $p$ that does not divide $\#(E(\QQ)_{\rm tors})$ (see the discussion in \S\ref{somr section}). Thus, taking account of Remark \ref{RemarkEquivalence}, the study of Conjecture \ref{mtconj} will in principal allow one to verify the Mazur-Tate Conjecture modulo its components at the finitely many primes $p$ that do not satisfy the conditions in Hypothesis \ref{hyp1}(iii) and (iv). This means, in particular, that our methods always neglect the $2$-primary and $3$-primary components of the Mazur-Tate Conjecture (whenever they arise). A further restriction on our approach that appears difficult to remove is that, following Hypothesis \ref{hyp1}(ii), we can only consider the situation in which the respective conductors of the curve $E$ and field $F$ are coprime. \end{remark}



\section{Review of the Generalized Perrin-Riou Conjecture}\label{sec GPR}

In this section we review the construction of `Bockstein regulator' elements from \cite{bks4} and then formulate a natural `algebraic' analogue of (a special case of) the Generalized Perrin-Riou Conjecture studied in loc. cit.    

\subsection{The Bockstein regulator} We keep the notations in the previous subsection. 
We denote the $p$-adic Tate module of $E$ by $T$ and set $V:=\QQ_p\otimes_{\ZZ_p} T$. We set
$$\Sigma := S \cup \{ p\} =\{\ell \mid pmN\}.$$
We assume $r :={\rm rank}(E(\QQ))> 0$ in the rest of this section. In this case, the natural localization map $ E(\QQ)\otimes_\ZZ\QQ_{p} \to  E_1(\QQ_{p})\otimes_{\ZZ_p}\QQ_p$ is surjective so that $H^1(\ZZ_\Sigma,V)=H^1_{f}(\QQ,V)$ and hence the Kummer map induces an isomorphism
\begin{equation}\label{kt iso} E(\QQ)\otimes_\ZZ\ZZ_p \simeq H^1(\ZZ_\Sigma,T).\end{equation}
(See \cite[(2.2.1)]{bks4}.) 

Write $D(\ZZ_p)$ for the derived category of $\ZZ_p$-modules and let $F/\QQ$ be a finite abelian extension with Galois group $G$. (For the moment, we do not need to assume Hypothesis \ref{hyp1}.) Then the tautological exact sequence
$$0 \to I_F/I_F^2 \to \ZZ_p[G]/I_F^2 \to \ZZ_p \to 0$$
combines with the canonical projection isomorphism $\rgamma(\cO_{F,\Sigma},T)\otimes^{\DL}_{\ZZ_p[G]} \ZZ_p \simeq \rgamma(\ZZ_{\Sigma},T)$ in $D(\ZZ_p)$ to give a canonical exact triangle in $D(\ZZ_p)$  
\[ \rgamma(\ZZ_{\Sigma},T)\otimes^{\DL}_{\ZZ_p} I_F/I_F^2 
\to \rgamma(\cO_{F,\Sigma},T)\otimes^{\DL}_{\ZZ_p[G]} \ZZ_p[G]/I_F^2 
\to \rgamma(\ZZ_\Sigma,T) \to . 
\]
We consider the composite homomorphism 
\begin{equation}\label{beta}
\beta=\beta_F: E(\QQ)\otimes_\ZZ\ZZ_p \xrightarrow{(\ref{kt iso})} H^1(\ZZ_\Sigma, T) \to 
H^2(\ZZ_\Sigma,T) \otimes_{\ZZ_p} I_F/I_F^2
\end{equation}
where the last arrow denotes $(-1)$-times the connecting homomorphism that is induced by the above exact triangle (cf. \cite[\S 2.3]{bks4}) and then define a homomorphism 
$${\rm Boc}_F: \bigl({\bigwedge}_\ZZ^r E(\QQ)\bigr)\otimes_\ZZ\ZZ_p \to E(\QQ) \otimes_\ZZ {\bigwedge}_{\ZZ_p}^{r-1}H^2(\ZZ_\Sigma,T) \otimes_{\ZZ_p} I_F^{r-1}/I_F^r$$
by
$$x_1\wedge \cdots \wedge x_r \mapsto \sum_{i=1}^r (-1)^{i+1}x_i \otimes \beta(x_1)\wedge \cdots \wedge \beta(x_{i-1}) \wedge \beta(x_{i+1})\wedge \cdots \wedge \beta(x_r).$$

Since $p$ is fixed we write $\QQ_\infty$ for the cyclotomic $\ZZ_p$-extension of $\QQ$ and $\QQ_n$ for each natural number  $n$ for the unique subfield of $\QQ_\infty$ of degree $p^n$ over $\QQ$. We note that the augmentation ideal $I_\infty$ of $\ZZ_{p}[[\Gal(\QQ_{\infty}/\QQ)]]$ identifies with the inverse limit over the ideals $I_{\QQ_n}$ (with respect to the natural  projection maps) and so by applying the above construction with $F$ taken to be each field $\QQ_n$ and then passing to the inverse limit over $n$ we 
 obtain a canonical homomorphism of $\ZZ_p$-modules 
$$
{\rm Boc}_{\QQ_{\infty}}: \bigl({\bigwedge}_\ZZ^r E(\QQ)\bigr) \otimes_\ZZ \ZZ_{p} \to 
E(\QQ) \otimes_\ZZ {\bigwedge}_{\ZZ_p}^{r-1}H^2(\ZZ_\Sigma,T) 
\otimes_{\ZZ_p} I_{\infty}^{r-1}/I_{\infty}^r.
$$

Next we note that the natural exact sequence
$$0 \to H^2(\ZZ_\Sigma,V)^{\ast} \to H^1(\ZZ_\Sigma,V) \to 
E_1(\QQ_{p}) \otimes_{\ZZ_p} \QQ_{p} \to 0$$
(cf. \cite[(2.2.2)]{bks4}) gives rise to a canonical composite isomorphism of $\QQ_p$-spaces 
\begin{align}\label{can comp ident} \varepsilon_\Sigma: \bigl({\bigwedge}_\ZZ^{r} E(\QQ)\bigr) \otimes_\ZZ\QQ_p \simeq&\, \bigl({\bigwedge}_{\QQ_p}^{r}H^1(\ZZ_\Sigma,V)\bigr) \otimes_{\QQ_p}\bigl(E_1(\QQ_p) \otimes_{\ZZ_p} \QQ_{p}\bigr)^\ast\\
\simeq&\, {\bigwedge}_{\QQ_p}^{r-1}H^2(\ZZ_\Sigma,V)^\ast\notag\\
\simeq&\, \bigl({\bigwedge}_{\ZZ_p}^{r-1}H^2(\ZZ_\Sigma,T)\bigr)^\ast\otimes_{\ZZ_p}\QQ_p \notag\end{align}
in which the first isomorphism is induced by (\ref{kt iso}) and the isomorphism $E_1(\QQ_p) \otimes_{\ZZ_p} \QQ_{p}$ $ \simeq \QQ_p$ induced by the logarithm ${\rm log}_\omega$ associated to the fixed N\'eron differential $\omega$, and the last isomorphism is the obvious identification.   

Any choice of basis element $\bm{x}$ of the $\ZZ$-module ${\bigwedge}_\ZZ^{r} E(\QQ)_{\rm tf}$ therefore gives rise to an  isomorphism of $\QQ_p$-spaces 
\[ \varepsilon_\Sigma(\bm{x}): \bigl({\bigwedge}_{\ZZ_p}^{r-1}H^2(\ZZ_\Sigma,T)\bigr)\otimes_{\ZZ_p}\QQ_p \simeq \QQ_p\]
and hence to a `Bockstein regulator' element  
\begin{eqnarray*} R^{\rm Boc}  := \bigl((1\otimes \varepsilon_\Sigma(\bm{x})\otimes 1)\circ ({\rm Boc}_{\QQ_\infty}\otimes_{\ZZ_p}\QQ_p)\bigr)(\bm{x}\otimes 1) \in E(\QQ) \otimes_{\ZZ} I_{\infty}^{r-1}/I_{\infty}^r \otimes_{\ZZ_{p}} \QQ_{p} \end{eqnarray*}
that is easily seen to be independent of the choice of $\bm{x}$. 


\subsection{Kato's zeta elements}

In the following, we fix a finite abelian $p$-extension $F$ of $\QQ$ that satisfies Hypothesis \ref{hyp1} (cf. Remark \ref{pextension}). 

We recall that if Kato's zeta element (Euler system)
$$z_F=z_{F,\Sigma} \in H^1(\cO_{F,\Sigma},V)$$
is normalized as in \cite[Def. 6.8]{bss2}, then for every $\chi$ in $\widehat{G}$ one has  
\begin{eqnarray}\label{katoexp}
\sum_{\sigma \in G}\chi(\sigma) \exp_\omega^\ast(\sigma z_F) =\frac{L_\Sigma (E,\chi,1)}{\Omega^+},
\end{eqnarray}
where $\exp_\omega^\ast: H^1(\cO_{F,\Sigma},V) \to H^1_{/f}(F_p,V) \to F_p$ denotes the dual exponential map that is associated to the fixed N\'eron differential $\omega$ (see \cite[Th. 6.6 and 9.7]{katoasterisque}). 

Now Hypothesis \ref{hyp1}(iii) implies that $H^1(\cO_{F,\Sigma},T)$ is torsion-free and hence identifies with a sublattice of  $H^1(\cO_{F,\Sigma},V)$. Further, since $E$ has good reduction at $p$ (by Hypothesis \ref{hyp1}(iii)) and the 
$G_\QQ$-representation $E[p]$ is irreducible (by Hypothesis \ref{hyp1}(iv)) one knows that   
$z_F$ belongs to the lattice $H^1(\cO_{F,\Sigma},T)$ (cf. \cite[Rem. 6.9 and \S6.4]{bss2}). 

\subsection{The Generalized Perrin-Riou Conjecture}

In our earlier article \cite[\S4.2]{bks4} we constructed a `wild Kolyvagin 
derivative' element
$$\kappa  \in H^1(\ZZ_\Sigma,T)\otimes_{\ZZ_p} I_{\infty}^{r-1}/I_{\infty}^r \simeq E(\QQ) \otimes_\ZZ I_\infty^{r-1}/I_\infty^r,$$
with the property that for each natural number $n$ one has 
\begin{equation}\label{key kappa} \iota_{\QQ_{n}}(\varrho_n(\kappa))= 
\sum_{\sigma \in \Gamma_n}\sigma(z_{\QQ_{n}}) \otimes \overline{\sigma}^{-1} \end{equation}
%
Here  $\Gamma_n$ denotes $\Gal(\QQ_{n}/\QQ)$, 
\[ \varrho_n: H^1(\ZZ_\Sigma,T)\otimes_{\ZZ_p} I_{\infty}^{r-1}/I_{\infty}^r\to H^1(\ZZ_\Sigma,T)\otimes_{\ZZ_p} I_{\QQ_n}^{r-1}/I_{\QQ_n}^r\]
is the natural projection map,  
\[ \iota_{\QQ_{n}}:
H^1(\ZZ_\Sigma, T) \otimes_{\ZZ_p} I_{\QQ_{n}}^{r-1}/I_{\QQ_{n}}^{r} \to 
H^1({\mathcal O}_{\QQ_{n},\Sigma}, T) 
\otimes_{\ZZ_{p}} \ZZ_p[\Gamma_n]/I_{\QQ_{n}}^{r}\]
is the homomorphism induced by the restriction map $H^1(\ZZ_\Sigma, T)\to H^1({\mathcal O}_{\QQ_{n},\Sigma}, T)$ and inclusion $I_{\QQ_n}^{r-1}/I_{\QQ_n}^r \to \ZZ_p[\Gamma_n]/I_{\QQ_{n}}^{r}$ and $\overline{\sigma}$ denotes the image of $\sigma$ in $\ZZ_p[\Gamma_n]/I_{\QQ_n}^r$.  


The central conjecture of \cite{bks4} predicts an explicit formula for $\kappa$. To state this conjecture we regard the $\Sigma$-truncated analytic Birch and Swinnerton-Dyer constant $\mathcal{L}^{\rm an}_\Sigma = \mathcal{L}_\Sigma(E)^{\rm an}$ defined in Remark \ref{truncate remark} as an element of $\CC_p$ by means of a fixed embedding $\RR \hookrightarrow \CC_p$. 

\begin{conjecture}[Generalized Perrin-Riou Conjecture for $E$ and $\QQ_{\infty}/\QQ$]\label{GPR20} One has 
\[\kappa= \mathcal{L}^{\rm an}_\Sigma\cdot R^{\rm Boc} \,\,\text{ and }\,\,R^{\rm Boc}\not= 0 \]
in $E(\QQ) \otimes_{\ZZ} \bigl(I_{\infty}^{r-1}/I_{\infty}^r \otimes_{\ZZ_{p}} \CC_{p}\bigr)$. \end{conjecture}


\begin{remark}\label{GPR r=1 rem} If $r = 1$, then 
 ${\rm Boc}_{\QQ_{\infty}}$ is the identity automorphism of $E(\QQ)\otimes_\ZZ \ZZ_p$ and one can check that $R^{\rm Boc} = 
 {\rm log}_\omega(x)\cdot x$ for any choice of a generator $x$ of $E(\QQ)_{\rm tf}$. This fact implies directly that, if $r =1$, then $R^{\rm Boc}\not= 0$ and can also be used to show that, in this case, Conjecture \ref{GPR20} coincides with the conjecture formulated by Perrin-Riou in \cite{PR} (for details of this deduction see \cite[Prop. 4.15, Rem. 4.13 and Rem. 4.10]{bks4}). \end{remark}

\begin{remark}\label{GPR non-zero rem} The Generalized Perrin-Riou Conjecture \cite[Conj. 2.12]{bks4} does not predict non-vanishing of the Bockstein regulator for $E$ relative to every extension $F/\QQ$. However, the Bockstein regulator for $E$ relative to $\QQ_\infty/\QQ$ can be explicitly described in terms of classical $p$-adic regulators that have been  conjectured not to vanish (see \cite[Th. 5.6 and 5.11]{bks4}) and this accounts for the prediction $R^{\rm Boc}\not= 0$ in Conjecture \ref{GPR20}. \end{remark} 

To state an `algebraic' analogue of Conjecture \ref{GPR20} we use the $\Sigma$-truncated algebraic Birch and Swinnerton-Dyer constant 
$\mathcal{L}_\Sigma^{\rm alg}$ 
defined in (\ref{euler factor}).    


\begin{conjecture}[Algebraic Generalized Perrin-Riou Conjecture for $E$ and $\QQ_\infty/\QQ$] \label{GPR2} One has 
$$\kappa=  \mathcal{L}^{\rm alg}_\Sigma\cdot R^{\rm Boc} \,\,\text{ and }\,\,R^{\rm Boc}\not= 0$$
in $E(\QQ) \otimes_{\ZZ} \bigl(I_{\infty}^{r-1}/I_{\infty}^r \otimes_{\ZZ_{p}} \QQ_{p}\bigr)$.
\end{conjecture}

The next result follows directly from the explicit interpretation of the Birch and Swinnerton-Dyer Conjecture described in Remark \ref{truncate remark}. 

\begin{lemma}\label{equivalence} If $E$ validates the Birch and Swinnerton-Dyer Conjecture over $\QQ$, then Conjecture \ref{GPR20} is equivalent to Conjecture \ref{GPR2}.\end{lemma}

\section{The main result}\label{state}

\subsection{Statement of the main result and the deduction of Theorem \ref{TheoremIntroduction}}\label{somr section}

The main result that we shall prove in the remainder of this article is the following. 

\begin{theorem}\label{MainResult}
Assume that the data $E, F$ and $p$ satisfies Hypothesis \ref{hyp1}. Assume also that $r:={\rm rank}(E(\QQ)) > 0$ and that the Algebraic Generalized Perrin-Riou Conjecture (Conjecture \ref{GPR2}) is valid. Then the Mazur-Tate Conjecture (Conjecture \ref{mtconj}) is valid. 
\end{theorem}

In the next section we shall describe the basic strategy that will be used to prove this result. However, before doing so, 
we first explain how it implies Theorem \ref{TheoremIntroduction}. 

At the outset we note that, since Theorem \ref{TheoremIntroduction} assumes $E$ validates the Birch and Swinnerton-Dyer Conjecture over $\QQ$, Remark \ref{rem bsd} allows us to assume $r > 0$ and Lemma \ref{equivalence} implies that for every prime $p$ the Generalized Perrin-Riou Conjecture for $E$ relative to the cyclotomic $\ZZ_p$-extension of $\QQ$ is equivalent to Conjecture \ref{GPR2}. We further note that, under these hypotheses, $L_S(E,1)$ vanishes and so the interpolation property (\ref{char}) implies that the modular element $\theta_{F,S}$ belongs to $I(G)\otimes_\ZZ\QQ$.  

To proceed we write $\ZZ'$ for the subring of $\QQ$ generated by the inverse of     
\[ D:= 6\cdot \#(E(\QQ)_{\rm tors}).\]
Then it is known that $\theta_{F,S}$ belongs to $\ZZ'[G]$ and hence to $I(G)\otimes_\ZZ\ZZ'$ and the central conjecture formulated by Mazur and Tate in \cite{MT} further predicts that $\theta_{F,S}$ belongs to $I(G)^r\otimes_\ZZ \ZZ'$ and that  its image in the quotient module $\bigl(I(G)^r/I(G)^{r+1}\bigr)\otimes_\ZZ \ZZ'$ is equal to a precise multiple of the Mazur-Tate regulator. 

Now $\theta_{F,S}$ belongs to the lattice $I(G)^r\otimes_\ZZ \ZZ'$ if and only if it belongs to $I(G)^r\otimes_\ZZ\ZZ_\ell$ for every prime $\ell$ that does not divide $D$. Further, since $r > 0$, the module $I(G)^r/I(G)^{r+1}$ is isomorphic to a quotient of the finite group ${\rm Sym}^r(G)$ and so decomposes as a direct sum 
\begin{equation*}\label{aug iso}  \frac{I(G)^r}{I(G)^{r+1}} \simeq \bigoplus_{\ell \mid \#G} \frac{(I(G)\otimes_\ZZ\ZZ_\ell)^r}{(I(G)\otimes_\ZZ\ZZ_\ell)^{r+1}}\end{equation*}
where $\ell$ runs over all prime divisors of $\#G$. In particular, if $\ell$ does not divide $\#G$, then $I(G)^r\otimes_\ZZ\ZZ_\ell = I(G)\otimes_\ZZ\ZZ_\ell$ and so the containment $\theta_{F,S}\in I(G)^r\otimes_\ZZ\ZZ_\ell$ follows directly from the observations made above.  

These facts combine to imply that it is enough to verify the Mazur-Tate Conjecture after replacing 
 $\ZZ'$ by $\ZZ_p$ for each prime divisor $p$ of $\#G$ that does not divide $D$. 
 
In addition, for any such prime $p$, the hypotheses (a), (b) and (c) in Theorem \ref{TheoremIntroduction} together imply that the data $E, F$ and $p$ verify all of the conditions in Hypothesis \ref{hyp1}. Thus, if the Generalized Perrin-Riou Conjecture for $E$ relative to the cyclotomic 
$\ZZ_p$-extension of $\QQ$ is valid, then the validity of the Mazur-Tate Conjecture after replacing $\ZZ'$ by $\ZZ_p$ follows directly from Theorem \ref{MainResult} (and Lemma \ref{equivalence}).  
 
This completes the proof of Theorem \ref{TheoremIntroduction}.

\subsection{Determinantal zeta elements}
 
In this section, we introduce (in Definition \ref{det zeta def}) a natural notion of `determinantal zeta element' and discuss the key role that such elements will play in the proof of Theorem \ref{MainResult}. In particular, in this way we shall establish an important reduction step in the proof of the latter result. 

We remark at the outset that determinantal zeta elements have implicitly played a key role in several of our earlier articles. For example, the `determinantal zeta element for $\GG_m$' is the central object of study in \cite{bks1} (where it is simply referred to as the `zeta element for $\mathbb{G}_m$') and is also used to formulate and study a natural main conjecture of higher rank Iwasawa theory in \cite{bks2}. In addition, in the respective articles \cite{bks2-2} and \cite{bks4} we studied analogous elements in the determinant of the Galois cohomology of Tate twists $\ZZ_{p}(j)$ (for arbitrary integers $j$) and of the $p$-adic Tate modules of elliptic curves and, more recently, Kataoka \cite{kataoka} has made a systematic study of determinantal zeta elements in the setting of general Galois representations.

To be more precise, we henceforth write $F_{\infty}$ for the cyclotomic $\ZZ_{p}$-extension of $F$. For each subfield $K$ of $F_{\infty}$ we also write $\Lambda_K$ for the algebra $\ZZ_{p}[[\Gal(K/\QQ)]]$ and regard the tensor product 
\[ T_{K}:= T \otimes_{\ZZ_p} \Lambda_K\]
as a module over $\ZZ_p[[G_\QQ]]\otimes_{\ZZ_p}\Lambda_K$ in the natural way. 

Then, as a first step in the proof of Theorem \ref{MainResult}, we shall use the equivariant theory of Kolyvagin and Stark systems in the setting of Kato's zeta elements in order to construct for each such field $K$ a `determinantal zeta element' $\fz_{K}$ in the linear dual of the $\Lambda_K$-determinant of the Galois cohomology of $T_K$ over $\ZZ_\Sigma$ (for details see Definition \ref{det zeta def}).

Next we prove (in Proposition \ref{main0}) that the validity of Conjecture \ref{GPR2} implies $\fz_\QQ$ coincides with an  `algebraic Birch and Swinnerton-Dyer element' 
$\eta^{\rm alg}$ that is defined using the approach of \cite{bks4}.  
 
Using this last observation, the proof of Theorem \ref{MainResult} is reduced to showing that the validity of Conjecture \ref{mtconj} is implied by the equality  $\fz_{\QQ} = \eta^{\rm alg}$ (see Theorem \ref{main}) and this deduction will then be proved in \S\ref{sec pf}.

To proceed we write $\Omega(K)$ for the set of subfields of $K$ that are of finite degree over $\QQ$, regarded as partially ordered by inclusion, and we use the compatibility under norm of Kato's zeta elements to define an element 
$$z_{K} = (z_{M,\Sigma})_{M \in \Omega(K)} 
\in H^1(\ZZ_{\Sigma},T_{K})
 = \varprojlim_{M \in \Omega(K)} H^1(\cO_{M,\Sigma},T).$$

In the sequel we also write ${\det}_{\Lambda_K}(C)$ for the 
$\Lambda_K$-equivariant determinant (in the sense of Knudsen and Mumford) of any perfect complex of $\Lambda_K$-modules $C$ and we set 
\[ {\det}_{\Lambda_K}^{-1}(C) := \Hom_{\Lambda_K}({\det}_{\Lambda_K}(C),\Lambda_K).\] 

Then it is well-known that the definition of $\Sigma$ ensures the Galois cohomology complex $\rgamma(\ZZ_{\Sigma},T_{K})$ is a perfect complex of $\Lambda_K$-modules. In Proposition \ref{comm1} below we will define a canonical homomorphism of $\Lambda_K$-modules 
$$\pi_K: {\det}_{\Lambda_K}^{-1}(\rgamma(\ZZ_{\Sigma},T_{K})) \to H^1(\ZZ_{\Sigma},T_{K}),$$
and by taking the limit of these homomorphisms over all $K$ in $\Omega(F_\infty)$ we thereby obtain a canonical homomorphism of 
$\Lambda_{F_{\infty}}$-modules 
$$\pi_{F_{\infty}}:
{\det}_{\Lambda_{F_{\infty}}}^{-1}(\rgamma(\ZZ_{\Sigma},T_{F_{\infty}})) \to 
H^1(\ZZ_{\Sigma},T_{F_{\infty}}).$$

The following observation regarding the link between the element $z_{F_\infty}$ and map $\pi_{F_\infty}$ relies on the equivariant theory of Euler, Kolyvagin and Stark systems and will play a key role in our approach. 

\begin{lemma}[Kataoka] \label{kataoka lemma}  If Hypothesis \ref{hyp1} is satisfied, then $z_{F_{\infty}}$ belongs to the image of $\pi_{F_{\infty}}$. 
\end{lemma}

\begin{proof} This has recently been proved in \cite{kataoka} and, for the reader's convenience, we sketch the argument. 

In \cite[Th. 3.12(ii)]{sbA}, the first and the third authors showed that, for each pair of natural numbers $n$ and $m$, there exists a natural isomorphism  of $\ZZ/p^m[\Gal(F_n/\QQ)]$-modules
$$\pi_{n,m}:{\det}_{\ZZ/p^m[\Gal(F_n/\QQ)]}^{-1}(\rgamma(\ZZ_\Sigma, T_{F_n}/p^m)) \xrightarrow{\sim} {\rm SS}_1(T_{F_n}/p^m),$$
where ${\rm SS}_1(-)$ denotes the module of Stark systems of rank one (see also Theorem 5.6 in \cite{kataoka}).

In addition, in \cite[Th. 5.2]{bss}, Sakamoto and the first and the third authors showed 
that there is a natural isomorphism  of $\ZZ/p^m[\Gal(F_n/\QQ)]$-modules
$${\rm Reg}: {\rm SS}_1(T_{F_n}/p^m) \xrightarrow{\sim} {\rm KS}_1(T_{F_n}/p^m),$$
where ${\rm KS}_1(-)$ denotes the module of Kolyvagin systems of rank one. 

Now, it is well-known that one obtains a Kolyvagin system from an Euler system, 
 and that Kato's Euler system yields a Kolyvagin system $ x_{n,m}$ for which one has  
$$(x_{n,m})_1= z_{F_n} \in H^1(\cO_{F_n,\Sigma},T/p^m)=H^1(\ZZ_\Sigma,T_{F_n}/p^m).$$

The required claim is therefore true since the homomorphism $\pi_{F_\infty}$ coincides, by its very construction, with the  inverse limit (over $n$ and $m$) of the composite maps 
\begin{multline*}{\det}_{\ZZ/p^m[\Gal(F_n/\QQ)]}^{-1}(\rgamma(\ZZ_\Sigma, T_{F_n}/p^m)) \xrightarrow{\pi_{n,m}} {\rm SS}_1(T_{F_n}/p^m)\\ \xrightarrow{{\rm Reg}} {\rm KS}_1(T_{F_n}/p^m) \xrightarrow{\kappa \mapsto \kappa_1} H^1(\ZZ_\Sigma,T_{F_n}/p^m).\end{multline*}
\end{proof}

It is clear that the $\Lambda_{F_{\infty}}$-module ${\det}_{\Lambda_{F_{\infty}}}^{-1}(\rgamma(\ZZ_{\Sigma},T_{F_{\infty}}))$ is free of rank one. Thus, since the annihilator of $z_{F_{\infty}}$ in $\Lambda_{F_\infty}$ vanishes (as a consequence of the argument in \cite[\S 6.1]{kataoka}), the result of Lemma \ref{kataoka lemma} implies that $\pi_{F_{\infty}}$ is injective, and hence that there exists a unique element 
$\fz_{F_{\infty}}$ of ${\det}_{\Lambda_{F_\infty}}^{-1}(\rgamma(\ZZ_{\Sigma},T_{F_{\infty}}))$ such that 
$$\pi_{F_{\infty}}(\fz_{F_{\infty}})=z_{F_{\infty}}.$$
 
This observation motivates us to make the following definition. 

\begin{definition}\label{det zeta def} For any subfield $K$ of $F_{\infty}$, the `($p$-adic) determinantal zeta element of $K$' is the image $\fz_{K}$ of $\fz_{F_{\infty}}$ under the 
canonical projection map  
$${\det}_{\Lambda_{F_\infty}}^{-1}(\rgamma(\ZZ_{\Sigma},T_{F_{\infty}})) \twoheadrightarrow 
{\det}_{\Lambda_K}^{-1}
(\rgamma(\ZZ_{\Sigma},T_{K})).$$
\end{definition}

Then the main observation we wish to make in this section is that Conjecture \ref{GPR2} implies an explicit formula for the determinantal zeta element of $\QQ$. 

To state this result we use the canonical `passage to cohomology' isomorphism 
\begin{eqnarray}\label{det bottom}
{\det}_{\ZZ_p}^{-1}(\rgamma(\ZZ_\Sigma,T)) \simeq 
\# H^2(\ZZ_\Sigma,T)_{\rm tors} \cdot {\bigwedge}_{\ZZ_p}^r 
H^1(\ZZ_\Sigma,T) \otimes_{\ZZ_p} {\bigwedge}_{\ZZ_p}^{r-1}H^2(\ZZ_\Sigma,T)_{\rm tf}^\ast. 
\end{eqnarray}

We also recall (from \cite[(2.5.2)]{bks4}) that, if we fix a basis element $\bm{x}$ of 
${\bigwedge}_\ZZ^{r} E(\QQ)_{\rm tf}$ and use the isomorphisms (\ref{kt iso}) and (\ref{can comp ident}), then the `algebraic Birch and Swinnerton-Dyer element'   
\begin{equation}\label{eta alg def}
\eta^{\rm alg} := \mathcal{L}^{\rm alg}_\Sigma\cdot (\bm{x}\otimes \varepsilon_\Sigma(\bm{x}))\end{equation}
belongs to the codomain of the isomorphism (\ref{det bottom}) and is independent of the choice of $\bm{x}$. 

In the sequel we can (and will) therefore use (\ref{det bottom}) to regard $\eta^{\rm alg}$ as an element of the lattice ${\det}_{\ZZ_p}^{-1}(\rgamma(\ZZ_\Sigma,T))$. 


\begin{proposition}\label{main0}
 If Conjecture \ref{GPR2} is valid, then $\fz_{\QQ}=\eta^{\rm alg}$.
\end{proposition}

\begin{proof} The key point in this argument is that for any real abelian extension $F/\QQ$, with $G = \Gal(F/\QQ)$, the map $\pi_F$ lies in a commutative diagram of the form
$$
\small
\xymatrix{
{\det}_{\ZZ_p[G]}^{-1}(\rgamma(\cO_{F,\Sigma},T)) \ar[r]^{\hskip 0.2truein \pi_F} \ar@{->>}[d]_{\nu_{F}}& H^1(\cO_{F,\Sigma},T) \ar[r]^{\hskip -0.3truein \cN_{F/\QQ}}& H^1(\cO_{F,\Sigma},T)\otimes_{\ZZ_p} \ZZ_p[G]/I_F^r \\
{\det}_{\ZZ_p}^{-1}(\rgamma(\ZZ_\Sigma,T)) \ar[r]_{(\ref{det bottom})\quad\quad\quad\quad}&  {\bigwedge}_{\ZZ_p}^r H^1(\ZZ_\Sigma,T) \otimes_{\ZZ_p} {\bigwedge}_{\ZZ_p}^{r-1}H^2(\ZZ_\Sigma,T)_{\rm tf}^\ast \ar[r]_{\quad\quad\quad {\rm Boc}'_{F}}& H^1(\ZZ_\Sigma,T)\otimes_{\ZZ_p} I_F^{r-1}/I_F^r. \ar[u]_{\iota_F}
}
$$
Here $\nu_{F}$ is the natural projection map, $\cN_{F/\QQ}$ is the `Darmon norm' that sends each element $x$ to the class of $\sum_{\sigma \in G}\sigma(x)\otimes \sigma^{-1}$, $\iota_F$ is the map induced by the restriction map $H^1(\ZZ_\Sigma,T)\to H^1(\cO_{F,\Sigma},T)$ and inclusion $I^{r-1}/I^r \to \ZZ_p[G]/I^r$ and ${\rm Boc}'_{F}$ denotes the composite homomorphism 
\begin{align*} &\,{\bigwedge}_{\ZZ_p}^r H^1(\ZZ_\Sigma,T) \otimes_{\ZZ_p} {\bigwedge}_{\ZZ_p}^{r-1}H^2(\ZZ_\Sigma,T)_{\rm tf}^\ast \\
\xrightarrow{{\rm Boc}_F\otimes 1 }&\,\bigl(H^1(\ZZ_\Sigma,T) \otimes_{\ZZ_p} {\bigwedge}_{\ZZ_p}^{r-1}H^2(\ZZ_\Sigma,T) \otimes_{\ZZ_p} I_F^{r-1}/I_F^r\bigr)\otimes_{\ZZ_p} {\bigwedge}_{\ZZ_p}^{r-1}H^2(\ZZ_\Sigma,T)_{\rm tf}^\ast\\
\to&\, H^1(\ZZ_\Sigma,T) \otimes_{\ZZ_p} I_F^{r-1}/I_F^r\end{align*}
where the last map is induced by the natural isomorphism 
\[ \bigl({\bigwedge}_{\ZZ_p}^{r-1}H^2(\ZZ_\Sigma,T)_{\rm tf}\bigr) \otimes_{\ZZ_p}\bigl({\bigwedge}_{\ZZ_p}^{r-1}H^2(\ZZ_\Sigma,T)_{\rm tf}^\ast\bigr) \simeq \ZZ_p. \]
In addition, the commutativity of the above diagram follows from the same explicit computation that verifies commutativity of the diagram (7.4.1) in \cite{bks4}. 

We also note that the map $\iota_F$ is injective. Indeed, this follows easily from the facts that $H^1(\ZZ_\Sigma,T)$ is $\ZZ_p$-free and that $H^1(\ZZ_\Sigma,T)$ identifies with the submodule $H^1(\cO_{F,\Sigma},T)^G$ of $G$-invariant elements in $H^1(\cO_{F,\Sigma},T)$ (since $H^0(\ZZ_\Sigma,T)$ vanishes).

In addition, a comparison of the definitions of $\eta^{\rm alg}$ and $R^{\rm Boc}$ combines with a direct calculation to show that the homomorphism  
\[ {\rm Boc}_\infty': {\bigwedge}_{\ZZ_p}^r H^1(\ZZ_\Sigma,T) \otimes_{\ZZ_p} {\bigwedge}_{\ZZ_p}^{r-1}H^2(\ZZ_\Sigma,T)_{\rm tf}^\ast \to H^1(\ZZ_\Sigma,T) \otimes_{\ZZ_p} I_\infty^{r-1}/I_\infty^r\]
given by the limit of the maps ${\rm Boc}_{\QQ_n}'$ for $n \ge 1$ sends $\eta^{\rm alg}$ to $\mathcal{L}^{\rm alg}_\Sigma\cdot R^{\rm Boc}$. 

In particular, if Conjecture \ref{GPR2} is valid, then the commutativity of the above diagram (for each field $F = \QQ_n$)  combines with the property (\ref{key kappa}) of the element $\kappa$ to imply that 
\[ {\rm Boc}_\infty'(\fz_\QQ) = \mathcal{L}^{\rm alg}_\Sigma \cdot R^{\rm Boc} = {\rm Boc}_\infty'(\eta^{\rm alg}).\]
Thus, since ${\rm Boc}_\infty'$ is injective (as Conjecture \ref{GPR2} predicts $R^{\rm Boc}\not=0$), one has $\fz_\QQ = \eta^{\rm alg}$, as required. 
\end{proof}

\begin{remark}
The equality $\fz_{\QQ}=\eta^{\rm alg}$ that occurs in Proposition \ref{main0} implies the `refined Mazur-Tate conjecture' formulated in \cite[Conj. 2.19]{bks4} for any abelian $p$-extension $F/\QQ$ of the form considered in loc. cit. This is because if one applies  
the commutative diagram in the proof of Proposition \ref{main0} to $F$ and 
the element $\fz_{F}$, then one obtains an equality $\cN_{F/\QQ}(z_{F})= \iota_F({\rm Boc}_{F}'(\eta^{\rm alg})),$ which can be shown to agree precisely with the formulation of \cite[Conj. 2.19]{bks4}. 
\end{remark}

In view of Proposition \ref{main0}, it is clear that Theorem \ref{MainResult} is a direct consequence of the following result (that will be proved in the next section).

\begin{theorem}\label{main}
If $\fz_{\QQ}=\eta^{\rm alg}$, then Conjecture \ref{mtconj} is valid. 
\end{theorem}


\begin{remark} If $r=1$, then one can use the isomorphism (\ref{det bottom}) to regard $\eta^{\rm alg}$ as an element of $H^1(\ZZ_\Sigma,T)$ and hence consider the possibility that  
\begin{eqnarray}\label{alg PR}
z_\QQ \stackrel{?}{=}\eta^{\rm alg},
\end{eqnarray}
where $z_{\QQ}=z_{\QQ, \Sigma}$ is Kato's zeta element in $H^1(\ZZ_\Sigma,T)$. This displayed equality is a natural `algebraic' variant of Perrin-Riou's conjecture (as discussed in \cite[Conj. 2.8 and Prop. 2.10]{bks4}) and is easily seen to be equivalent to the equality 
$\fz_{\QQ}=\eta^{\rm alg}$ assumed in Theorem \ref{main}. Hence, if $r=1$, then Theorem \ref{main} implies that the $p$-primary component of the  Mazur-Tate Conjecture is directly implied by the algebraic Perrin-Riou Conjecture given by (\ref{alg PR}). 
\end{remark}

%
%
%

\section{The proof of Theorem \ref{main} and the deduction of Corollary \ref{cor1}}\label{sec pf}

\subsection{Construction of the modular element}\label{constr}

We first review the construction of the modular element $\theta_{F,S}$ from 
Kato's zeta element $z_F$ (see \cite{kuriharass} by the second author, \cite{otsuki} by Otsuki 
and especially \cite{ota} by Ota). 

As is shown in \cite[Lem. 5.5]{ota}, we have an isomorphism
\begin{eqnarray}\label{formal}
E_1(F_p) \xrightarrow{\sim} \cO_F\otimes_{\ZZ}\ZZ_p
\end{eqnarray}
defined by
$$c \mapsto \left(1-\frac{a_p}{p}{\rm Fr}_p +\frac 1p {\rm Fr}_p^2 \right)\log_\omega(c),$$
where $a_p:=p+1-\# E(\FF_p)$, ${\rm Fr}_p \in G$ is the arithmetic Frobenius at $p$, and $\log_\omega$ is the formal logarithm $E_1(F_p) \to p \cO_F \otimes_\ZZ \ZZ_p$ associated to $\omega$. 
(The assumptions in loc. cit. are satisfied, since Hypothesis \ref{hyp1} ensures that $p$ is unramified in $F$ and $p \nmid 6N\# E(\FF_p)$.) 
We define 
$$c_F \in E_1(F_p)$$
to be the element corresponding to 
$${\rm Tr}_{\QQ(\zeta_m)/F}(\zeta_m) \in \cO_F$$
under the isomorphism (\ref{formal}). 
We regard $c_F$ as an element of $H^1_f(F_p,T)$ by means of the Kummer map $E_1(F_p) \to H^1_f(F_p,T)$.

We write 
\begin{equation*}\label{cup}
(-,-)_F: H^1_f(F_p,T) \times H^1_{/f}(F_p,T) \to H^2(F_p,\ZZ_p(1)) \simeq \bigoplus_{\fp \mid p}\ZZ_p \xrightarrow{(a_\fp)_\fp \mapsto \sum_\fp a_\fp} \ZZ_p\end{equation*}
for the pairing induced by the cup product (after identifying $T$ with $T^\ast(1)$ by means of the Weil pairing). By composing it with the localization map $H^1(\cO_{F,\Sigma},T) \to H^1_{/f}(F_p,T)$, we obtain a pairing 
$$(-,-)_F: H^1_f(F_p,T) \times H^1(\cO_{F,\Sigma},T) \to \ZZ_p, $$
which we denote by the same symbol. 

\begin{proposition}\label{modularkato} In $\ZZ_p[G]$ one has 
$$\theta_{F,S}=\sum_{\sigma \in G}(c_F, \sigma z_F)_F \sigma^{-1}.$$
\end{proposition}

\begin{proof}
By (\ref{char}), it is sufficient to show
$$\sum_{\sigma \in G}(c_F,\sigma z_F)_F \chi^{-1}(\sigma)=\tau_m(\chi) \frac{L_{S } (E,\chi^{-1},1)}{\Omega^+} $$
for every $\chi \in \widehat G$. By the computation in \cite[(5.10)]{ota}, the left hand side is equal to
$$\left(\sum_{\sigma \in G}\log_\omega(\sigma c_F)\chi(\sigma) \right)\times \left( \sum_{\sigma \in G} \exp_\omega^\ast(\sigma z_F) \chi^{-1}(\sigma) \right).$$
Since we have
$$\sum_{\sigma \in G}\log_\omega(\sigma c_F)\chi(\sigma)=\left(1-\frac{a_p}{p}\chi^{-1}({\rm Fr}_p) +\frac 1p \chi^{-2}({\rm Fr}_p)\right)^{-1} \tau_m(\chi)$$
by \cite[Prop. 5.8]{ota}, the desired equality follows from Kato's formula (\ref{katoexp}). 
\end{proof}


\subsection{Bloch-Kato Selmer complexes}



In this section we shall use the local point $c_F \in E_1(F_p)$ that was constructed in \S\ref{constr} to construct an isomorphism of $\ZZ_p[G]$-modules of the form 
$$\varphi_F: {\det}_{\ZZ_p[G]}^{-1}(\rgamma(\cO_{F,\Sigma},T)) \xrightarrow{\sim} {\det}_{\ZZ_p[G]}^{-1}(\rgamma_f(F,T)).$$

To do this we note first that if $\ell \in \Sigma$ and $\ell \neq p$ (i.e., $\ell \mid mN$), then Hypothesis \ref{hyp1}(iii) implies $E(F_\ell)[p]$ vanishes and hence that the complex $\rgamma_{/f}(F_\ell,T)$ is acyclic. In this case, therefore, there exists a canonical isomorphism of $\ZZ_p[G]$-modules. 
\begin{eqnarray}\label{det ell}
{\det}_{\ZZ_p[G]}^{-1}(\rgamma_{/f} (F_\ell,T)) \simeq {\det}_{\ZZ_p[G]}^{-1}(0) = \ZZ_p[G]. 
\end{eqnarray}

Next we note that Hypothesis \ref{hyp1}(iii) implies the group 
$$H^2(F_p,T)\simeq E(F_p)[p^\infty]^\vee$$
vanishes, and hence that there is a natural identification
\begin{eqnarray}\label{at p}
\rgamma_{/f}(F_p,T) =H^1_{/f}(F_p,T)[-1]=E_1(F_p)^\ast[-1].
\end{eqnarray}

In addition, since $F/\QQ$ is tamely ramified, the $\ZZ_p[G]$-module $E_1(F_p)\simeq \cO_F\otimes_\ZZ \ZZ_p$ is isomorphic to $\ZZ_p[G]$ (by the Hilbert-Speiser theorem) and is generated by the element $c_F$. The identification (\ref{at p}) therefore induces 
 an isomorphism of $\ZZ_p[G]$-modules
\begin{eqnarray}\label{cmap}
{\det}_{\ZZ_p[G]}^{-1}(\rgamma_{/f}(F_p,T)) = {\det}_{\ZZ_p[G]}^{-1}(E_1(F_p)^\ast[-1]) = E_1(F_p)^\ast \simeq \ZZ_p[G],
\end{eqnarray}
in which the last map sends $c_F^\ast$ to $1$. 

Now we note that there is a canonical exact triangle of $\ZZ_p[G]$-modules
\begin{eqnarray}\label{exact}
\rgamma_f(F,T) \to \rgamma(\cO_{F,\Sigma},T) \to \bigoplus_{\ell \in \Sigma} \rgamma_{/f}(F_\ell,T),
\end{eqnarray}
and that the above argument shows that (each direct summand of) the last term in this triangle is a perfect complex of $\ZZ_p[G]$-modules. Since $\rgamma(\cO_{F,\Sigma},T)$ is also a perfect complex of $\ZZ_p[G]$-modules, this exact triangle therefore implies that $\rgamma_f(F,T)$ is a perfect complex of $\ZZ_p[G]$-modules and also induces a  canonical isomorphism of $\ZZ_p[G]$-modules

\begin{equation}\label{det triang}
{\det}_{\ZZ_p[G]}^{-1}(\rgamma(\cO_{F,\Sigma},T)) \simeq {\det}_{\ZZ_p[G]}^{-1}(\rgamma_f(F,T)) \otimes_{\ZZ_p[G]} \bigotimes_{\ell \in \Sigma} {\det}_{\ZZ_p[G]}^{-1}(\rgamma_{/f}(F_\ell,T)).
\end{equation}

Then, upon combining this isomorphism with the maps (\ref{det ell}) (for each $\ell\in \Sigma\setminus\{p\}$) and (\ref{cmap})  we obtain the desired isomorphism 
 $\varphi_F:{\det}_{\ZZ_p[G]}^{-1}(\rgamma(\cO_{F,\Sigma},T)) \to {\det}_{\ZZ_p[G]}^{-1}(\rgamma_f(F,T))$. 

The key property of this map $\varphi_F$ that we shall use in the sequel is established by the following result. 

\begin{proposition}\label{comm1}
There exist canonical maps $\pi_F$ and $\pi_{F,f}$ that lie in a commutative diagram of $\ZZ_p[G]$-modules
$$
\xymatrix{
{\det}_{\ZZ_p[G]}^{-1}(\rgamma(\cO_{F,\Sigma},T)) \ar[r]^{\quad\quad\pi_F} \ar[d]_{\varphi_F} & H^1(\cO_{F,\Sigma},T) \ar[d]^{\sum_{\sigma \in G} (c_F, \sigma(\cdot))_F \sigma^{-1}} \\
{\det}_{\ZZ_p[G]}^{-1}(\rgamma_f(F,T)) \ar[r]_{\quad\quad \pi_{F,f}} & \ZZ_p[G].
}
$$  
\end{proposition}

\begin{proof}
One knows 
that the complex $\rgamma(\cO_{F,\Sigma},T)$ is represented by
$$P_F \xrightarrow{\psi} Q_F,$$
where $P_F$ and $Q_F$ are free $\ZZ_p[G]$-modules of rank $d$ and $d-1$ respectively for some $d>r$, and $P_F$ is placed in degree one (see the proof of \cite[Th. 4.3]{bks4}). Since $\rgamma_{/f}(F_\ell, T)$ is acyclic if $\ell \in S =\Sigma \setminus \{p\}$, we see by the triangle (\ref{exact}) that
$$\rgamma_f(F,T)\simeq {\rm Cone}\left(\rgamma(\cO_{F,\Sigma},T) \xrightarrow{{\rm loc}_p} \rgamma_{/f}(F_p,T) \right)[-1].$$
By the definition of mapping cones and (\ref{at p}), we see that
$$\rgamma_f(F,T) = \left[  P_F \xrightarrow{\psi_f}  E_1(F_p)^\ast \oplus Q_F\right],$$
where $P_F$ is placed in degree one and $\psi_f$ is given by
$$\psi_f(a) := ({\rm loc}_p(a) , -\psi(a)).$$

Fix a basis $\{c_2,\ldots,c_d\}$ of  $Q_F$. For each $i$ with $2 \leq i\leq d$, we set
$$\psi_i := c_i^\ast \circ \psi : P_F \to \ZZ_p[G] .$$
Similarly, for each $i$ with $1\leq i\leq d$, we set
$$\psi_{f,i}:=c_i^\ast \circ \psi_f: P_F\to \ZZ_p[G],$$
where $c_1:=c_F^\ast \in E_1(F_p)^\ast$. 

We define 
$$\pi_F: {\det}_{\ZZ_p[G]}^{-1}(\rgamma(\cO_{F,\Sigma},T)) = {\bigwedge}_{\ZZ_p[G]}^d P_F \otimes_{\ZZ_p[G]} {\bigwedge}_{\ZZ_p[G]}^{d-1}Q_F^\ast \to P_F$$
by
$$\pi_F(a \otimes c_2^\ast  \wedge\cdots \wedge c_d^\ast) := \left( {\bigwedge}_{2\leq i\leq d}\psi_i \right)(a).$$
One checks that this is well-defined and the image is contained in $H^1(\cO_{F,\Sigma},T)$. 

Similarly, we define
$$\pi_{F,f}: {\det}_{\ZZ_p[G]}^{-1}(\rgamma_f(F,T)) = {\bigwedge}_{\ZZ_p[G]}^d P_F \otimes_{\ZZ_p[G]}{\bigwedge}_{\ZZ_p[G]}^d (E_1(F_p)^\ast \oplus Q_F)^\ast \to \ZZ_p[G]$$
by
$$\pi_{F,f}(a \otimes c_1^\ast \wedge \cdots \wedge c_d^\ast):=\left({\bigwedge}_{1\leq i\leq d}\psi_{f,i}\right)(a).$$

We now prove the commutativity of the diagram in the claim. One checks by definition that $\varphi_F$ is explicitly given by
$$\varphi_F(a \otimes c_2^\ast \wedge \cdots \wedge c_d^\ast)=a\otimes c_1^\ast \wedge \cdots \wedge c_d^\ast.$$
Noting that $\psi_{f,i}=-\psi_i$ for $2\leq i \leq d$, we have
$$\left({\bigwedge}_{1\leq i\leq d}\psi_{f,i}\right)(a)=(-1)^{d-1} \psi_{f,1}\circ \left( {\bigwedge}_{2\leq i \leq d}\psi_{f,i}\right)(a)=\psi_{f,1}\circ \left({\bigwedge}_{2\leq i \leq d}\psi_i \right)(a).$$
Hence we have
$$\pi_{F,f}\circ \varphi_F = \psi_{f,1} \circ \pi_F. $$
So it is sufficient to prove 
$$\psi_{f,1}(a)=\sum_{\sigma\in G}(c_F, \sigma(a))_F \sigma^{-1}$$
for any $a \in H^1(\cO_{F,\Sigma},T)$. By the definition of $\psi_{f,1}$, we have
$$\psi_{f,1}(a)=c_1^\ast({\rm loc}_p(a)).$$
Here ${\rm loc}_p(a)$ is an element of $H^1_{/f}(F_p,T)$, which we identify with $E_1(F_p)^\ast$ via the pairing $(-,-)_F$, and $c_1^\ast $ is by definition the map $E_1(F_p)^\ast \to \ZZ_p[G]$ sending $c_F^\ast$ to 1. It is now easy to check
$$c_1^\ast({\rm loc}_p(a))=\sum_{\sigma \in G}( c_F,\sigma(a) )_F \sigma^{-1},$$
which completes the proof. 
\end{proof}

One can define maps $\pi_{F_{n}}$ and $\pi_{F_{\infty}}$ similarly for the $n$-th layer $F_{n}$ and 
the cyclotomic $\ZZ_{p}$-extension $F_{\infty}$. 
 In particular, since the following diagram (in which both vertical arrows are the natural projection maps) 
$$
\xymatrix{
{\det}_{\Lambda_{F_\infty}}^{-1}(\rgamma(\ZZ_{\Sigma},T_{F_{\infty}})) \ar[r]^{\quad\quad\pi_{F_{\infty}}} 
\ar[d]_{} & H^1(\ZZ_{\Sigma},T_{F_\infty})
 \ar[d]^{} \\
{\det}_{\ZZ_p[G]}^{-1}(\rgamma(\ZZ_{\Sigma},T_{F}))= 
{\det}_{\ZZ_p[G]}^{-1}(\rgamma(\cO_{F,\Sigma},T)) \ar[r]_{\quad\quad\quad\quad\quad\quad\quad\quad \pi_{F}} & 
H^1(\cO_{F,\Sigma},T),
}
$$
is commutative, one has  
$$\pi_F(\fz_F)=z_F.$$
We can therefore deduce the following consequence of Propositions \ref{modularkato} and \ref{comm1}. 

\begin{corollary}\label{corkato} In $\ZZ_p[G]$ one has 
$$\theta_{F,S}=\pi_{F,f}(\varphi_F(\fz_{F})).$$
\end{corollary}

\subsection{Bockstein maps for Selmer complexes}\label{section bmc}

As we have seen before, Hypothesis \ref{hyp1} ensures that the Bloch-Kato Selmer complex $\rgamma_f(F,T)$ is a perfect complex of $\ZZ_p[G]$-modules which is acyclic outside degrees one and two and satisfies the base change property 
$$\rgamma_f(F,T)\lotimes_{\ZZ_p[G]}\ZZ_p \simeq \rgamma_f(\QQ,T).$$
Using this, we obtain an exact triangle
$$\rgamma_f(\QQ,T)\lotimes_{\ZZ_p}I/I^2 \to \rgamma_f(F,T)\lotimes_{\ZZ_p[G]}\ZZ_p[G]/I^2 \to \rgamma_f(\QQ,T),$$
and hence an induced connecting homomorphism in its long exact cohomology sequence
$$H^1_f(\QQ,T) \to H^2(\rgamma_f(\QQ,T)\otimes_{\ZZ_p} I/I^2)=H^2_f(\QQ,T)\otimes_{\ZZ_p}I/I^2,$$
where the displayed equality follows directly from the fact that $\rgamma_f(\QQ,T)$ is acyclic in degrees greater than two. 

It is convenient to use $(-1)$-times the above `Bockstein map', which we denote by
$$\beta_{f}:H^1_f(\QQ,T) \to H^2_f(\QQ,T) \otimes_{\ZZ_p}I/I^2.$$
Upon combining this map with the canonical isomorphisms 
$$H^1_f(\QQ,T)\simeq \ZZ_p\otimes_\ZZ E(\QQ) \text{ and }H^2_f(\QQ,T)\simeq {\rm Sel}_p(E/\QQ)^\vee,$$
where ${\rm Sel}_p(E/\QQ)$ denotes the classical $p$-Selmer group, we obtain a composite map (which we denote by the same symbol)
\begin{eqnarray*}\label{betaf}
\beta_{f}: \ZZ_p \otimes_\ZZ E(\QQ) \to {\rm Sel}_p(E/\QQ)^\vee \otimes_{\ZZ_p}I/I^2 \twoheadrightarrow (\ZZ_p\otimes_\ZZ E(\QQ))^\ast \otimes_{\ZZ_p}I/I^2,
\end{eqnarray*}
where the second map is the natural surjection. We recall that the associated pairing
$$\langle -,-\rangle_F: (\ZZ_p \otimes_\ZZ E(\QQ)) \times (\ZZ_p\otimes_\ZZ E(\QQ)) \to I/I^2; \ (x,y) \mapsto \beta_{f}(y)(x)$$
is known to coincide with the Mazur-Tate pairing (\ref{mt pair}) (this is proved by Macias Castillo and the first author in \cite[Th. 10.3]{bmc}).

Let
$${\rm Reg}_F: \ZZ_p\otimes_\ZZ \left({\bigwedge}_\ZZ^r E(\QQ) \otimes_\ZZ {\bigwedge}_\ZZ^r E(\QQ) \right) \to I^r/I^{r+1}$$
be the map defined by
$${\rm Reg}_F(x_1\wedge\cdots \wedge x_r \otimes y_1\wedge \cdots \wedge y_r):=\det(\langle x_i,y_j \rangle_F)_{1\leq i,j\leq r}.$$

\begin{proposition}\label{comm2}
Let $\pi_{F,f}: {\det}_{\ZZ_p[G]}^{-1}(\rgamma_f(F,T)) \to \ZZ_p[G]$ be the map constructed in Proposition \ref{comm1}. Then the following claims are valid.
\begin{itemize}
\item[(i)] The image of $\pi_{F,f}$ is contained in $I^r$. 
\item[(ii)] The following diagram is commutative:
$$
\xymatrix{
{\det}_{\ZZ_p[G]}^{-1}(\rgamma_f(F,T)) \ar@{->>}[dd]_{\nu_f} \ar[r]^{\pi_{F,f}} & I^r  \ar@{->>}[rd]& \\
 & & I^r/I^{r+1} \\
 {\det}_{\ZZ_p}^{-1}(\rgamma_f(\QQ,T)) \ar[r]_{\hskip -0.5truein\pi_{\QQ,f}} &  \ZZ_p\otimes_\ZZ \left({\bigwedge}_\ZZ^r E(\QQ) \otimes_\ZZ {\bigwedge}_\ZZ^r E(\QQ) \right), \ar[ru]_{\quad\quad{\rm Reg}_F}& 
}
$$
where both $\nu_f$ and the unlabelled arrow denote the natural projection maps and $\pi_{\QQ, f}$ is induced by the canonical isomorphism
$${\det}_{\ZZ_p}^{-1}(\rgamma_f(\QQ,T))\simeq {\det}_{\ZZ_p}(H^1_f(\QQ,T)) \otimes_{\ZZ_p} {\det}_{\ZZ_p}^{-1}(H^2_f(\QQ,T)).$$
\end{itemize}
\end{proposition}

\begin{proof}
This follows directly from the argument of \cite[Th. 3.3.7]{bst} after fixing the data $(R,C,J,a,a')$ in loc. cit. to be $(\ZZ_p, \rgamma_f(F,T), G, 0,r)$. 
\end{proof}

\subsection{The algebraic Birch and Swinnerton-Dyer element}

Let
$$\varphi_\QQ: {\det}_{\ZZ_p}^{-1}(\rgamma(\ZZ_\Sigma,T)) \xrightarrow{\sim} {\det}_{\ZZ_p}^{-1}(\rgamma_f(\QQ,T))$$
be the isomorphism induced by $\varphi_F$, which fits into the commutative diagram
\begin{eqnarray}\label{comm3}
\xymatrix{
{\det}_{\ZZ_p[G]}^{-1} (\rgamma(\cO_{F,\Sigma},T)) \ar[r]^{\hskip 0.1truein\varphi_F}\ar@{->>}[d]_{\nu} & {\det}_{\ZZ_p[G]}^{-1} (\rgamma_f(F,T)) \ar@{->>}[d]^{\nu_f}\\
{\det}_{\ZZ_p}^{-1}(\rgamma(\ZZ_\Sigma ,T)) \ar[r]_{\varphi_\QQ}& {\det}_{\ZZ_p}^{-1} (\rgamma_f(\QQ,T)).
}
\end{eqnarray}

In the following result we use the element $\eta^{\rm alg}$ of ${\det}_{\ZZ_p}^{-1}(\rgamma(\ZZ_\Sigma ,T))$ defined in (\ref{eta alg def}), the $S$-truncated algebraic Birch and Swinnerton-Dyer constant $\mathcal{L}^{\rm alg}_S\in \QQ^\times$ defined in (\ref{euler factor}) and the Mazur-Tate regulator $R_F\in I^r/I^{r+1}$ defined in (\ref{MT reg def}). 

\begin{lemma}\label{lemalg} In $I^r/I^{r+1}$ one has 
$$({\rm Reg}_F \circ \pi_{\QQ,f}\circ \varphi_\QQ) ( \eta^{\rm alg})= (-1)^{\nu(m)}\mathcal{L}^{\rm alg}_S\cdot R_F.$$
\end{lemma}

\begin{proof} The inverse of the isomorphism $\varepsilon_\Sigma$ from (\ref{can comp ident}) induces a composite isomorphism of $\QQ_p$-spaces  
%
\begin{align*}\psi_\QQ : {\det}_{\ZZ_p}^{-1}(\rgamma(\ZZ_\Sigma,T))\otimes_{\ZZ_p}\QQ_p \simeq&\,\left(\QQ_p \otimes_\ZZ {\bigwedge}_\ZZ^r E(\QQ) \right)\otimes_{\QQ_p} {\bigwedge}_{\QQ_p}^{r-1}H^2(\ZZ_\Sigma,V)^\ast\\
\xrightarrow{1\otimes \varepsilon_\Sigma^{-1}}&\, \QQ_p\otimes_\ZZ \left({\bigwedge}_\ZZ^r E(\QQ)\otimes_\ZZ {\bigwedge}_\ZZ^r E(\QQ)\right)\end{align*}
in which the first isomorphism is induced by (\ref{det bottom}). This isomorphism is such that   
\begin{equation}\label{eta def cons} \psi_\QQ(\eta^{\rm alg}) = \mathcal{L}^{\rm alg}_\Sigma\cdot (\bm{x}\otimes \bm{x})\end{equation}
for any choice of basis element $\bm{x}$ of ${\bigwedge}_\ZZ^{r} E(\QQ)_{\rm tf}$ (where the displayed equality follows directly from the explicit definition of $\eta^{\rm alg}$). 


The composite of $\varphi_\QQ$ and the map $\pi_{\QQ,f}$ that occurs in the lower row of the diagram in Proposition \ref{comm2}(ii) also induces a similar isomorphism of $\QQ_p$-spaces 
$$\pi_{\QQ,f}\circ \varphi_\QQ: {\det}_{\ZZ_p}^{-1}(\rgamma(\ZZ_\Sigma,T))\otimes_{\ZZ_p}\QQ_p \simeq \QQ_p\otimes_\ZZ \left({\bigwedge}_\ZZ^r E(\QQ)\otimes_\ZZ {\bigwedge}_\ZZ^r E(\QQ)\right).$$
In addition, if we write $\N_{F/\QQ}: E_1(F_p)\to E_1(\QQ_p)$ for the natural norm map, then the relation
$$\frac{\# E(\FF_p)}{p}\cdot{\rm log}_\omega(\N_{F/\QQ}(c_F))={\rm Tr}_{\QQ(\zeta_m)/\QQ}(\zeta_m)=(-1)^{\nu(m)}$$
can be used to show that 
$$\psi_\QQ = (-1)^{\nu(m)}\frac{\# E(\FF_p)}{p} \cdot (\pi_{\QQ,f}\circ \varphi_\QQ).$$
 
From the equality (\ref{eta def cons}), one can therefore deduce that 
$$ (\pi_{\QQ,f}\circ \varphi_\QQ)(\eta^{\rm alg})= (-1)^{\nu(m)} \left(\frac{\# E(\FF_p)}{p}\right)^{-1}\mathcal{L}^{\rm alg}_\Sigma\cdot (\bm{x}\otimes \bm{x}) = (-1)^{\nu(m)} \mathcal{L}^{\rm alg}_S\cdot (\bm{x}\otimes \bm{x}),$$
where the last equality is valid because both $\Sigma = S\cup \{p\}$ and $p\notin S$ and hence 
\[  \left(\frac{\# E(\FF_p)}{p}\right)^{\!\!-1}\!\!\mathcal{L}^{\rm alg}_\Sigma = \mathcal{L}^{\rm alg}\!\left(\frac{\# E(\FF_p)}{p}\right)^{\!\!-1}\!\!\left(\prod_{\ell \in \Sigma} 
\frac{\# E^{\rm ns}(\FF_\ell)}{\ell}\right) = \mathcal{L}^{\rm alg}\!\left(\prod_{\ell \in S} 
\frac{\# E^{\rm ns}(\FF_\ell)}{\ell}\right) = \mathcal{L}^{\rm alg}_S.\] 

The above equality in turn implies that 
\[ ({\rm Reg}_F \circ \pi_{\QQ,f}\circ \varphi_\QQ) ( \eta^{\rm alg}) = (-1)^{\nu(m)} \mathcal{L}^{\rm alg}_S\cdot {\rm Reg}_F(\bm{x}\otimes \bm{x})\]
and this implies the claimed result since the definition of $R_F$ implies directly that it is equal to ${\rm Reg}_F(\bm{x}\otimes \bm{x})$.  
\end{proof}

\subsection{Completion of the proof of Theorem \ref{main}}\label{completion section}

Upon combining the equality $\theta_{F,S}= \pi_{F,f}(\varphi_F(\fz_{F}))$ from Corollary \ref{corkato} with the commutativity of the diagram in Proposition \ref{comm2}(ii), one derives an equality
\begin{eqnarray}\label{theta1}
\theta_{F,S}=({\rm Reg}_F\circ \pi_{\QQ,f}\circ  \nu_f \circ \varphi_F)(\fz_{F}) \text{ in }I^r/I^{r+1}.
\end{eqnarray}

In addition, if one assumes $\fz_{\QQ}=\eta^{\rm alg}$, then the commutative diagram (\ref{comm3}) 
implies that  
$$(\nu_f \circ \varphi_F)(\fz_{F})=\varphi_\QQ(\fz_{\QQ})=\varphi_\QQ(\eta^{\rm alg})$$
and hence, by Lemma \ref{lemalg}, that 
\begin{equation}\label{theta2}
({\rm Reg}_F\circ \pi_{\QQ,f}\circ  \nu_f \circ \varphi_F)(\fz_{F}) = (-1)^{\nu(m)}\mathcal{L}^{\rm alg}_S\cdot R_F.
\end{equation}

Since the equalities (\ref{theta1}) and (\ref{theta2}) combine to imply the equality $\theta_{F,S}= (-1)^{\nu(m)}\mathcal{L}^{\rm alg}_S\cdot R_F$ that is predicted by Conjecture \ref{mtconj}, this argument therefore completes the proof of Theorem \ref{main} (and hence also, by Proposition \ref{main0}, of Theorem \ref{MainResult}), as required.



\subsection{The deduction of Corollary \ref{cor1}}\label{deduction cor1 section}

We finally explain the deduction of Corollary \ref{cor1} from the proof of Theorem \ref{MainResult}. To do this we assume the hypotheses of Corollary \ref{cor1} (but no longer require that $\#G$ is a prime power). 

Then, under these hypotheses, the main result of Jetchev, Skinner and Wan in \cite{JSW} implies that the quotient 
\[ u := \mathcal{L}^{\rm an}_S/\mathcal{L}^{\rm alg}_S\]
is a (non-zero) rational number that is coprime to every prime divisor of $\#G$. 

From Remark \ref{truncate remark} we also know that $u =1$ if and only if $E$ validates the Birch and Swinnerton-Dyer Conjecture over $\QQ$. 

Hence, from the discussion of \S\ref{somr section}, the result of Corollary \ref{cor1} will follow if we can show that for each prime divisor $p$ of $\#G$ the displayed equality in Conjecture \ref{mtconj} is valid after one multiplies the right hand side by $u$ (which acts invertibly on $I/I^2 \simeq G$). 

To do this we fix such a prime $p$. Then, since, by assumption, $E$ has square-free conductor and supersingular reduction at $p$, the validity of the Generalized Perrin-Riou Conjecture for $E$ and $\QQ_{\infty}/\QQ$ (Conjecture \ref{GPR20}) follows directly from Remark \ref{GPR r=1 rem} and the result \cite[Th. 2.4(iv)]{buyuk perrin} of B\"uy\"ukboduk. From the  commutative diagram in the proof of Proposition \ref{main0} it therefore follows that $\fz_\QQ = u\cdot \eta^{\rm alg}$. 

The latter equality then combines with the argument of \S\ref{completion section} to imply that the equality (\ref{theta2}) is unconditionally valid provided that one multiplies its right hand side by $u$ and, by comparing this fact to the 
equality (\ref{theta1}), we deduce that the displayed equality in Conjecture \ref{mtconj} is also unconditionally valid after one multiplies its right hand side by $u$, as required.  

This completes the proof of Corollary \ref{cor1}.


\begin{thebibliography}{99999999}
%
%
%
%
%
%






\bibitem{BD} M. Bertolini, H. Darmon,
\newblock Kato's Euler system and rational points on elliptic curves I: a $p$-adic Beilinson formula,
\newblock Isr. J. Math. {\bf 199}(1) (2014) 163-188.

%
%

\bibitem{bleymodular} W. Bley,
\newblock The equivariant Tamagawa number conjecture and modular symbols,
\newblock Math. Ann. {\bf 356} no.1 (2013) 179-190.


%
%
%
%
%
%
%
%
%
%
%
%
%
%
%
%
%
%
%
%
%
%
%
%
%
%
%
%

%
%

\bibitem{bmc} D. Burns, D. Macias Castillo,
\newblock On refined conjectures of Birch and Swinnerton-Dyer type for Hasse-Weil-Artin $L$-series,
\newblock preprint, arXiv:1909.03959.



\bibitem{bks1} D. Burns, M. Kurihara, T. Sano,
\newblock On zeta elements for $\mathbb{G}_m$,
\newblock Doc. Math. \textbf{21} (2016) 555-626.

\bibitem{bks2} D. Burns, M. Kurihara, T. Sano,
\newblock On Iwasawa theory, zeta elements for $\mathbb{G}_m$ and the equivariant Tamagawa number conjecture,
\newblock Algebra \& Number Theory {\bf 11} (2017) 1527-1571.

\bibitem{bks2-2} D. Burns, M. Kurihara, T. Sano,
\newblock On Stark elements of arbitrary weight and their $p$-adic families,
\newblock Advanced Studies in Pure Mathematics {\bf 86}, 
\newblock Development of Iwasawa Theory -- the Centennial of K. Iwasawa's Birth, (2020) 113-140.

\bibitem{bks4} D. Burns, M. Kurihara, T. Sano,
\newblock On derivatives of Kato's Euler system for elliptic curves, 
\newblock preprint, arXiv:1910.07404.


\bibitem{bss} D. Burns, R. Sakamoto, T. Sano,
\newblock On the theory of higher rank Euler, Kolyvagin and Stark systems, II: the general theory,
\newblock preprint, arXiv:1805.08448.

\bibitem{bss2} D. Burns, R. Sakamoto, T. Sano,
\newblock On the theory of higher rank Euler, Kolyvagin and Stark systems, III: applications,
\newblock preprint, arXiv:1902.07002.

\bibitem{sbA} D. Burns, T. Sano,
\newblock On the theory of higher rank Euler, Kolyvagin and Stark systems,
\newblock to appear in Int. Math. Res. Not.

\bibitem{bst} D. Burns, T. Sano, K.-W. Tsoi,
\newblock On higher special elements of $p$-adic representations,
\newblock to appear in Int. Math. Res. Not. 


\bibitem{buyuk perrin} K. B\"uy\"ukboduk,
\newblock Beilinson-Kato and Beilinson-Flach elements, Coleman-Rubin-Stark classes, Heegner points and a Conjecture of Perrin-Riou,
\newblock Advanced Studies in Pure Mathematics {\bf 86}, 
\newblock Development of Iwasawa Theory -- the Centennial of K. Iwasawa's Birth, (2020) 141-193.






\bibitem{BPS} K. B\"uy\"ukboduk, R. Pollack, S. Sasaki,
\newblock $p$-adic Gross-Zagier formula at critical slope and a conjecture of Perrin-Riou - I,
\newblock preprint, arXiv:1811.08216v3.




%
%
%
%
%
%
%
%
%




\bibitem{darmoneuler} H. Darmon, 
\newblock Euler systems and refined conjectures of Birch and Swinnerton-Dyer type, 
\newblock Contemp. Math. {\bf 165} (1994) 265-276.


%
%
%
%




%
%
%
%
%
%
%
%
%
%
%
%


%






%
%
%
%
%
%
%




%
%
%




%
%
%
%
%

\bibitem{JSW} D. Jetchev, C. Skinner, X. Wan,
\newblock The Birch and Swinnerton-Dyer Formula for elliptic curves of analytic rank one,
\newblock Camb. J. Math. {\bf 5} (2017) 369-434.
%
%
%
%
%
%
%
%


\bibitem{kataoka} T. Kataoka,
\newblock Stark systems and equivariant main conjectures,
\newblock preprint. 


%
%
%
%

\bibitem{katoasterisque} K. Kato,
\newblock $p$-adic Hodge theory and values of zeta functions of modular forms,
\newblock Ast\'erisque, (295):ix, 117-290, 2004. Cohomologies $p$-adiques et applications arithm\'etiques. III.











\bibitem{kuriharass} M. Kurihara,
\newblock On the Tate Shafarevich groups over cyclotomic fields of an elliptic curve 
with supersingular reduction I,
\newblock Invent. math. {\bf 149} (2002) 195-224.




%
%
%

%
%
%
%
%
%
%
%
%
\bibitem{MT} B. Mazur, J. Tate,
\newblock Refined Conjectures of the Birch and Swinnerton-Dyer Type,
\newblock Duke Math. J.  {\bf 54} (1987) 711-750.


%
%
%
%
%
%




%
%
%

\bibitem{ota} K. Ota, 
\newblock Kato's Euler system and the Mazur-Tate refined conjecture of BSD type,
\newblock American J. Math., {\bf 140} no.2 (2018) 495-542. 

\bibitem{otsuki} R. Otsuki,
\newblock Construction of a Homomorphism Concerning Euler Systems for an Elliptic Curve,
\newblock Tokyo. J. Math. {\bf 32} no.1 (2009) 253-278.


\bibitem{PR} B. Perrin-Riou,
\newblock Fonctions $L$ $p$-adiques d'une courbe elliptique et points rationnels,
\newblock Ann. Inst. Fourier (Grenoble) {\bf 43} no.4 (1993) 945-995.


%
%
%
%


\bibitem{portillo} F. X. Portillo-Bobadilla,
\newblock Experimental Evidence on a Refined Conjecture of the BSD type,
\newblock Bol. Soc. Mat. Mex. {\bf 25} (2019) 529-541.



%
%
%
%




%


%
%
%

%
%
%
%
%
%
%


%
%
%
%
%
%
%
%
%
%
%
%
%
%
%


%
%





%
%
%
%




\end{thebibliography}
\end{document}